\documentclass[12pt]{article}

\RequirePackage{fancymath}

\usepackage{xcolor}
\usepackage{soul}
\usepackage{tikz}
\usetikzlibrary{patterns}
\usepackage{graphicx}
\usepackage{pgfplots}
\usepackage{latexsym,amsmath,graphics,amsthm,tikz,hyperref,float}
\usepackage[mathscr]{euscript}
\usepackage[affil-it]{authblk}
\usepackage{dsfont,pstricks}
\usepackage{latexsym,amsmath,amssymb,amscd,wrapfig}
\usepackage{hyperref,cleveref}
\usepackage{enumitem}
\usepackage{circuitikz}
\usepackage{bbm}

\makeatletter
\def\@maketitle{%
	\newpage
	\null
	\begin{center}%
		\let \footnote \thanks
		{\Large\bfseries \@title \par}%
		\vskip 1.5em%
		{\normalsize
			\lineskip .5em%
			\begin{tabular}[t]{c}%
				\@author
			\end{tabular}\par}%
		\vskip 0.8em%
		{\small \@date}%
	\end{center}%
	\par
	\vskip 1.5em}
\makeatother

\title{The sup-inf-completion of a Dedekind complete vector lattice}

\author{Eder Kikianty%
	\thanks{Email: \texttt{eder.kikianty@wits.ac.za}}}
\affil[1]{School of Mathematics, University of Witwatersrand, Private Bag 3, WITS 2050, South Africa}

\author[2]{Luan Naude%
	\thanks{Email: \texttt{luan.naude98@gmail.com}}}
\affil[2]{Department of Mathematics and Applied Mathematics, University of Pretoria, Private Bag X20, Hatfield 0028, South Africa}

\author{Mark Roelands%
	\thanks{Email: \texttt{m.roelands@math.leidenuniv.nl}}}
\affil[3]{Mathematical Institute, Leiden University, 2300 RA Leiden,
	The Netherlands}

\author[2]{Christopher Schwanke%
	\thanks{Email: \texttt{cmschwanke26@gmail.com}}}

\date{\today}

\begin{document}
	
	\maketitle
	
	\vspace{-.8cm}
	
	\begin{abstract}
		{We introduce the sup-inf-completion of a Dedekind complete vector lattice, an essentially unique extension in which every nonempty subset has both a supremum and an infimum. Since this completion is not a cone, we develop the more general framework of lattice stars, which provides the natural setting for its construction. We establish the fundamental properties of the sup-inf-completion, including a universal property, a representation theorem, and a characterization of its bands and band projections. As an application, we extend the Riemann integral on Dedekind complete $f$-algebras to Type I and Type II improper integrals. We conclude by showing that power series on universally complete vector lattices may be integrated term-by-term. }  
	\end{abstract}
	
	
	
	\section{Introduction}
	
	Completions have proven to be pivotal concepts for several important developments in vector lattice theory, and the idea behind them is simple: if you are working with a vector lattice devoid of desired properties, construct instead a suitable larger space, one of minimal size, that possesses the desired qualities.
	
	Of all the completions of an Archimedean vector lattice, its \textit{universal completion} is the largest one which is itself a vector lattice, and it is also known to be vector lattice isomorphic to $C^\infty(K)$. It has the properties that (i) every nonempty, bounded above subset has a supremum, and (ii) every nonempty pairwise disjoint subset has a supremum. At this point, any further extensions in the pursuit of even more ambitious order-theoretic properties must come with trade-offs in algebraic structure. Indeed, in the \textit{sup-completion} of a Dedekind complete vector lattice, first constructed in \cite{Donner}, \textit{every} nonempty subset possesses a supremum. However, this space is not a vector lattice (as it is not a vector space), but rather, more restrictively, a lattice cone. In \cite[Theorem 1]{TroitskyPolavarapu}, it was shown that the sup-completion is isomorphic as a lattice cone to
	\[
	\{f \in C(K, \overline{\bR}) \colon \text{ there is } x \in E \text{ such that } x \leq f \}.
	\]
	
	While the sup-completion of a Dedekind complete vector lattice has pro-ven extremely useful, the pursuit of obtaining order-theoretic generalizations of improper integrals of the first kind--that is, basic integrals of the type $\int_{-\infty}^{+\infty} f(x)\, dx$--on Dedekind complete $f$-algebras (a natural extension of the Riemann integral introduced in \cite{paper2}), one would need to extend the sup-completion even further, with yet again another trade-off. It is the principal endeavor of this paper to construct this extension. Specifically, we construct what we call the \textit{sup-inf-completion} of a Dedekind complete vector lattice, which is a space for which any nonempty subset possesses both a supremum and an infimum. The necessary compromise in structure here is that the sup-inf-completion is not a cone (as it does not have an addition structure), but rather a set equipped with a scalar multiplication, which we call a \textit{star}.
	
	In Section 3, we formally introduce stars, ordered stars, lattice stars, distributive lattice stars, as well as their \textit{semistar} counterparts. Using these fundamentals, we construct the essentially unique sup-inf-completion of a Dedekind complete vector lattice in Section 4, and we provide several important properties of this space in Section 5.
	
	As mentioned previously, a primary motivation for constructing the sup-inf-completion is to obtain Type I improper integrals on Dedekind complete $f$-algebras, which generalize integrals of elementary functions over the entire real number line. Section 6 of this paper is dedicated precisely to this purpose. On the topic of improper integrals, we also construct an improper integral of the second kind in Section 7. Finally, as an application, we integrate power series on universally complete vector lattices in Section 8, showing that they can be integrated on a term-by-term basis.
	
	We proceed with some preliminaries.
	
	\section{Preliminaries}
	For any unexplained terminology or basic results in vector lattices and $f$-algebras, we refer the reader to the standard texts \cite{AliprantisBurkinshaw,zaanen1, depagter,zaanen2}.
	
	Throughout this paper, $E$ denotes an Archimedean vector lattice. We proceed by recalling several notions and facts that are repeatedly used in the sequel. We denote the positive cone of $E$ by $E^+$, and its negative cone by $E^-$. A positive $x \in E^+$ is called a \emph{weak order unit} if the band generated by $x$ in $E$ is all of $E$, that is $B_x = \{x\}^{dd} = E$. Equivalently, for any $y \in E^+$ we have that $\sup_{n \in \bN} y \wedge nx = y$. For $x,y \in E$, the notation $x \ll y$ means that $y-x$ is a weak order unit in $E$. A band $B$ in $E$ with the property that $E = B \oplus B^d$ is called a \emph{projection band}, and there exists an order projection $\mathbb{P}_B$ from $E$ onto $B$. A vector lattice in which every band is a projection band is said to have the \emph{projection property}. In particular, every Dedekind complete vector lattice has the projection property. For $a \le b$ we denote the order interval consisting of all elements $x \in E$ between $a$ and $b$ by 
	\[
	[a,b]:= \{x \in E \colon a \le x \le b\},
	\]
	and $(a,b)$ is defined by 
	\[
	(a,b):= \{x \in E \colon a \ll x \ll b\}.
	\]
	Variations of the endpoints such as $(a,b]$ and $[a,b)$ are similarly defined. A function $f \colon \mathrm{dom}(f) \to E$ with $\dom(f) \subseteq E$ is said to be \emph{order bounded} if there exists $M \in E^+$ such that $|f(x)| \le M$ for every $x \in \mathrm{dom}(f)$. A vector lattice $E$ equipped with a bilinear product such that $xy \in E^+$ whenever $x,y \in E^+$ is called an \emph{$f$-algebra} if $x \wedge y = 0$ implies $xz \wedge y = zx \wedge y = 0$ for all $z \in E^+$. Note that every Archimedean $f$-algebra is automatically commutative by \cite[Theorem~10.1]{depagter}. The remainder of the preliminary section is divided into three subsections. The band decomposition of $E$ in terms of inequalities and the Riemann integral are concepts based on current and ongoing research by the authors, and the notion of the sup-completion of a Dedekind complete vector lattice plays a central role in this paper. For these reasons, the topics are treated separately.      
	
	\subsection{Band decompositions of $E$ from inequalities}
	
	Given a vector lattice $E$ with the projection property, for $x,y \in E$ we define the bands $B_{x < y}$, $B_{x \leq y}$, and $B_{x = y}$, as was done in \cite[Notation 2.4]{paper1} and use them to decompose $E$. 
	
	\begin{definition}
		Let $E$ be a vector lattice with the projection property, and let $x, y \in E$. We define $B_{x < y}$ to be the band generated by $(y - x)^+$, and write $B_{x \leq y}$ for $B_{y < x}^d$. Finally, $B_{x = y}$ is defined to be $B_{x \leq y} \cap B_{y \leq x}$.
	\end{definition}
	
	The following is \cite[Proposition 2.6]{paper1}.
	
	\begin{proposition}
		Let $E$ be a vector lattice with the projection property, and let $x, y \in E$. Then
		\[ E = B_{x < y} \oplus B_{y \leq x} = B_{x < y} \oplus B_{y < x} \oplus B_{x = y}. \] 
	\end{proposition}
	
	To motivate the notation, these bands can also be characterized as in \cite[Proposition 2.8]{paper1}.
	
	\begin{proposition} \label{p:inequality_decomposition}
		Let $E$ be a vector lattice with the projection property, and let $x, y \in E$. Then
		\begin{enumerate}[(i)]
			\item $B_{x \leq y}$ is the largest band for which the corresponding band projection $\bP$ satisfies $\bP(x) \leq \bP(y)$,
			\item $B_{x = y}$ is the largest band for which the corresponding band projection $\bP$ satisfies $\bP(x) = \bP(y)$, and
			\item $B_{x < y}$ is the largest band for which $\bP(y - x)$ is a weak order unit in $B$.
		\end{enumerate}
	\end{proposition}
	
	\subsection{The Riemann integral}
	
	We recall the definition of the Riemann integral for a locally band preserving function $f \colon [a, b] \to E$, where $[a, b]$ is an order interval in $E$. This construction was originally done in \cite{paper2}.

	\begin{definition}
		Let $E$ be a Dedekind complete vector lattice. A function $f \colon \dom(f) \to E$ with $\dom(f) \subseteq E$ is said to be \emph{locally band preserving} if whenever $\bP(x) = \bP(y)$ for some $x, y \in \dom(f)$ and any band projection $\bP$, we have that $\bP(f(x)) = \bP(f(y))$.
	\end{definition}
	
	\begin{definition}
		Let $E$ be a Dedekind complete $f$-algebra and let $a,b \in E$ with $a \le b$. Furthermore, let $f \colon [a, b] \to E$ be locally band preserving and order bounded. A \emph{partition} $P$ of $[a, b]$ is a totally ordered subset containing $a$ and $b$, i.e. $P = \{a = x_0, \dots, x_n = b \}$. The \emph{lower} and \emph{upper sums} of $f$ with respect to $P$ are given by
		\begin{align*}
			L(f, P) &:= \sum_{i = 1}^n \left( \inf_{x \in [x_{i-1}, x_i]} f(x) \right)(x_i - x_{i-1}), \text{ and} \\
			U(f, P) &:= \sum_{i = 1}^n \left( \sup_{x \in [x_{i-1}, x_i]} f(x) \right)(x_i - x_{i-1}).
		\end{align*}
		The \emph{lower} and \emph{upper integrals} of $f$ are defined by
		\begin{align*}
			L(f) & := \sup \{L(f, P) \colon P \text{ is a partition of } [a, b] \}, \text{ and} \\
			U(f) & := \inf \{ U(f, P) \colon P \text{ is a partition of } [a, b] \}.
		\end{align*}
		If $L(f) = U(f)$, we say that $f$ is \emph{(Riemann) integrable} over $[a, b]$ and write
		\[ \int_a^b f(x) dx := L(f) = U(f). \]
	\end{definition}

	\subsection{The sup-completion of a Dedekind complete vector lattice}
	
	We provide a condensed overview of the sup-completion of a Dedekind complete vector lattice, as introduced in \cite[Section 1]{Donner}.
	\begin{definition}
		A \emph{cone} $C$ is a commutative monoid with unit $\overline{0}$ that is equipped with a nonnegative scalar multiplication $\bR^+ \times C \to C$, $(\lambda, x) \mapsto \lambda x$, such that
		\begin{enumerate}[(i)]
			\item $\lambda (x + y) = \lambda x + \lambda y$ for $\lambda \in \bR^+$ and $x,y \in C$,
			\item $(\lambda + \mu) x = \lambda x + \mu x$, for $\lambda, \mu \in \bR^+$ and $x \in C$,
			\item $\lambda (\mu x) = (\lambda \mu) x$ for $\lambda \mu \in \bR^+$ and $x \in C$,
			\item $1x = x$ for $x \in C$, and
			\item $0x = \overline{0}$ for $x \in C$.
		\end{enumerate}
		
		If $C$ is partially ordered such that
		\begin{enumerate}[(i)] \addtocounter{enumi}{5}
			\item $x \leq y$ implies $x + z \leq y + z$ for $x, y, z \in C$, and
			\item $x \leq y$ implies $\lambda x \leq \lambda y$ for $\lambda \in \bR^+$,
		\end{enumerate}
		then $C$ is called an \emph{ordered cone}. If, additionally, $C$ is a lattice with respect to $\leq$, then $C$ is called a \emph{lattice cone}. A lattice cone $C$ is said to be \emph{Dedekind complete} if every nonempty subset of $C$ which is bounded above (respectively, bounded below) has a supremum (respectively, infimum) in $C$. In what follows, the element $\overline{0}$ is denoted simply by $0$, and $C_0$ denotes the set of invertible elements of $C$, that is, the elements of the cone $C$ that have an additive inverse.
	\end{definition}
	
	\begin{proposition}
		Given a Dedekind complete vector lattice $E$, there exists an essentially unique cone $C$ such that
		\begin{enumerate}[(i)]
			\item $C$ is Dedekind complete,
			\item $E = C_0$ with coinciding algebraic and order structures,
			\item for every $y \in C$, we have $y = \sup \{x \in E \colon x \leq y \}$,
			\item $z + (x \wedge y) = (z + x) \wedge (z + y)$ for $x \in E$, $y, z \in C$,
			\item if $x \in E$, $y \in C$, and $y \leq x$, then $y \in E$,
			\item $C$ has a largest element, and
			\item if $A, B \subseteq C$ are nonempty sets such that $\sup A = \sup B$, then we also have $\sup(A \wedge x) = \sup(B \wedge x)$ for all $x \in E$.
		\end{enumerate}
	\end{proposition}
	
	This cone $C$ is called the \emph{sup-completion} of $E$ and is denoted by $E^s$. Analogously, one can also construct an inf-completion of $E^i$ of $E$, which is lattice cone isomorphic to $E^s$ with its ordering reversed. 
	
	\begin{remark}
		By \cite[Proposition 3]{Azouzi}, any band projection $\bP$ on $E$ can be extended to a left-order continuous, additive, increasing, and positively homogeneous map from $E^s$ to $E^s$. We denote this extension by $\bP$ as well. The extension is given by
		\[ \bP(x) := \sup \{ \bP(y) \colon y \in E, y \leq x \}. \]
	\end{remark}
	
	\Cref{l:sup_inf_representation} is \cite[Lemma 11(ii, iii)]{AzouziSI1}, where the nets are assumed to be positive. The proofs hold without modification even if the nets are not positive.  We use \Cref{l:sup_inf_representation} to prove needed properties of extended band projections in \Cref{l:sup_completion_projection_distribute_sup_inf}.
	
	\begin{lemma} \label{l:sup_inf_representation}
		Let $(x_\alpha)$ and $(y_\beta)$ be two nets in $E^s$, and let $x := \sup_\alpha x_\alpha$ and $y := \sup_\beta y_\beta$. Then the following statements hold.
		\begin{enumerate}[(i)]
			\item $x \vee y = \sup_{\alpha, \beta} (x_\alpha \vee y_\beta)$.
			\item $x \wedge y = \sup_{\alpha, \beta} (x_\alpha \wedge y_\beta)$.
		\end{enumerate}
	\end{lemma}
	
	\begin{lemma} \label{l:sup_completion_projection_distribute_sup_inf}
		The extension $\bP$ of a band projection in $E$ to $E^s$ satisfies the following properties.
		\begin{enumerate}[(i)]
			\item $\bP(x \vee y) = \bP(x) \vee \bP(y)$ for any $x, y \in E^s$.
			\item $\bP(x \wedge y) = \bP(x) \wedge \bP(y)$ for any $x, y \in E^s$.
		\end{enumerate}
	\end{lemma}
	\begin{proof}
		Let $x, y \in E^s$. There are nets $(x_\alpha)$ and $(y_\beta)$ in $E$ such that $x_\alpha \uparrow x$ and $y_\beta \uparrow y$. By left-order continuity and using \Cref{l:sup_inf_representation}, we get that
		\[ \bP(x \vee y) = \sup_{\alpha, \beta} \bP(x_\alpha \vee y_\beta) = \sup_{\alpha, \beta} (\bP(x_\alpha) \vee \bP(y_\beta)) = \bP(x) \vee \bP(y),\]
		and, similarly, $\bP(x \wedge y) = \bP(x) \wedge \bP(y)$.
	\end{proof}
	
	An \emph{ideal} $F$ in $E^s$ is a lattice subcone with the property that if $x \in E^s$, $y \in F$, and $\abs{x} \leq \abs{y}$, then $x \in F$.
	
	\begin{proposition} \label{p:projection_image_ideal}
		Let $\bP$ be a band projection in $E$ extended to $E^s$. Then $\bP(E^s)$ is an ideal in $E^s$.
	\end{proposition}
	\begin{proof}
		It is clear that $\bP(E^s)$ is a lattice subcone of $E^s$ since $\bP$ is additive and positively homogeneous. Let $x \in E^s$. Then
		\[ \bP(x^+) = \bP(x \vee 0) = \bP(x) \vee 0 = \bP(x)^+, \]
		and, similarly, $\bP(x^-) = \bP(x)^-$. Thus, 
		\[ \bP(\abs{x}) = \bP(x^+ + x^-) = \bP(x^+) + \bP(x^-) = \bP(x)^+ + \bP(x)^- = \abs{\bP(x)}. \]
		Therefore, $x \in \bP(E^s)$ if and only if $\abs{x} \in \bP(E^s)$.
		
		Let $0 \leq x \leq y$ with $x \in E^s$ and $y \in \bP(E^s)$. We will show that $x \in \bP(E^s)$. Since $y \in \bP(E^s)$, then $y = \bP(u)$ for some $u \in E^s$. By left order continuity of $\bP$ and $\bP^d$,
		\[ \bP^d(y) = \bP^d(\bP(u)) = \sup \{ \bP^d (\bP(z)) \colon z \in E, z \leq u \} = 0. \]
		Since $\bP^d$ is increasing, we have that $0 \leq \bP^d(z) \leq \bP^d(y) = 0$ for all $z \in E, z \leq y$. Now,
		\begin{align*}
			x
			& = \sup \{ z \colon z \in E, z \leq x \} \\
			& = \sup \{ \bP(z) + \bP^d(z) \colon z \in E, z \leq x \} \\
			& = \sup \{ \bP(z) \colon z \in E, z \leq x \} \\
			& = \bP(x),
		\end{align*}
		where the final equality is by definition of the extension. Thus, $x \in \bP(E^s)$. 
	\end{proof}
	
	In \cite{AzouziSI1}, the authors introduce a multiplication structure on $E^s_+$. For $x, y \in E^s_+$, they define
	\[ xy := \sup \{vw \colon v, w \in E^+, \\ v \leq x, \\ w \leq y \}.\]
	
	This multiplication satisfies the following properties. Statements (i), (ii), and (iii) are \cite[Lemma 24]{AzouziSI1}, and statement (iv) is an immediate consequence of statement (ii).
	
	\begin{proposition} \label{p:properties_multiplication_sup_comp}
		Let $x, y, z \in E^s_+$. Then
		\begin{enumerate}[(i)]
			\item $x(y + z) = xy + xz$,
			\item $x(y \wedge z) = xy \wedge xz$,
			\item $x(y \vee z) = xy \vee xz$, and
			\item if $x \leq y$, then $xz \leq yz$.
		\end{enumerate}
	\end{proposition}
	
	\section{Stars}
	
	Our abstract definitions are selected by considering the structure of the extended real numbers $\overline{\bR}$, which should be the sup-inf-completion of $\bR$. We have a scalar multiplication and a lattice structure, but we want to avoid expressions of the form $\infty + (-\infty)$.
	
	\begin{definition} \label{D: lattice star}
		
		A \emph{star} $S$ is a set with a scalar multiplication $\bR \times S \to S$, $(\lambda, x) \mapsto \lambda x$ satisfying
		\begin{itemize}
			\item[(i)] $\lambda(\mu x)=(\lambda\mu) x$ for $\lambda,\mu\in\mathbb{R}, x\in S$,
			\item[(ii)] $1x=x$ for $x\in S$, and
			\item[(iii)] there exists $\bar{0}\in S$ such that $0x=\bar{0}$ for all $x\in S$.
		\end{itemize}
		
		The unique element $\bar{0}\in S$ is called the \emph{zero element} of $S$, and will be denoted $0$ from now on. For any $x \in S$, we write $-x$ to mean $(-1)x$.
	\end{definition}
	
	\begin{definition}
		A star $S$ equipped with a partial ordering is called an \emph{ordered star} if for $x, y \in S$ and $\lambda \in \bR^+$,
		\begin{enumerate}[(i)]
			\item $x \leq y$ implies $\lambda x \leq \lambda y$, and
			\item $x \leq y$ implies $-x \geq -y$.
		\end{enumerate}
		
		An ordered star $S$ is called a \emph{lattice star} if $x \vee y := \sup \{x, y\}$ exists for all $x, y \in S$---or, equivalently, if $x \wedge y := \inf \{x, y\}$ exists for $x, y \in S$.
	\end{definition}
	
	\begin{proposition}
		Let $S$ be a lattice star. Then for $x, y \in S$ and $\lambda \in \bR^+$,
		\[ \lambda (x \vee y) = (\lambda x) \vee (\lambda y) \quad \text{and} \quad \lambda(x \wedge y) = (\lambda x) \wedge (\lambda y), \]
		and
		\[ -\lambda(x \vee y) = (-\lambda x) \wedge (-\lambda y) \quad \text{and} \quad -\lambda(x \wedge y) = (-\lambda x) \vee (-\lambda y). \]
	\end{proposition}
	\begin{proof}
		Assume $\lambda \neq 0$. Since $x \vee y \geq x, y$, we have that $\lambda(x \vee y) \geq \lambda x, \lambda y$. Suppose $u \geq \lambda x, \lambda y$ for some $u \in S$. Then $\lambda^{-1} u \geq x, y$ so that $\lambda^{-1} u \geq x \vee y$, and hence $u \geq \lambda(x \vee y)$. Therefore, $\lambda (x \vee y) = (\lambda x) \vee (\lambda y)$. The proofs of the other equalities are similar.
	\end{proof}
	
	\begin{definition}
		A lattice star $S$ is called a \emph{distributive lattice star} if for all $x, y, z \in S$,
		\[ x \wedge (y \vee z) = (x \wedge y) \vee (x \wedge z). \]
		$S$ is a distributive lattice star if and only if for all $x, y, z \in S$,
		\[ x \vee (y \wedge z) = (x \vee y) \wedge (x \vee z). \]
	\end{definition}
	
	Unlike in vector lattices, the distributivity property does not hold in general.
	
	\begin{example} 
		Consider $\mathbb{R}^2$, and take the diagonal and anti-diagonal, i.e.
		\[ S = \{(\lambda, \lambda) \colon \lambda \in \bR\} \cup \{(\lambda, -\lambda) \colon \lambda \in \bR\}. \]
		The set $S$ with the usual coordinatewise ordering and scalar multiplication is a lattice star. To find the supremum of two points $(x,y), (z,w) \in S$, consider the intersection of the standard cones emanating from the points, that is, 
		\[
		\{(a,b) \in \mathbb{R}^2\colon a \ge x,\ b \ge y\} \cap \{ (a,b) \in \mathbb{R}^2 \colon a \ge z,\ b \ge w \}.
		\]
		The left-and-lowermost point of $S$ in this intersection is the supremum of $(x,y)$ and $(z,w)$. For example, consider $A := (-1,1)$ and $B := (-\frac{1}{2},-\frac{1}{2})$. The supremum of $A$ and $B$ is $C := (1,1)$, as depicted in the figure below.
		
		\[
		\resizebox{0.4\textwidth}{!}
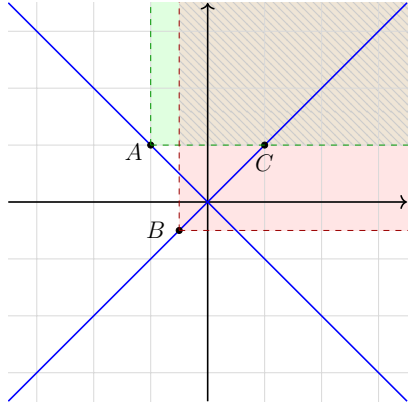
{
			\begin{tikzpicture}{H}
				
				\clip (-3.5,-3.5) rectangle (3.5,3.5);
				
				\fill[green!20, opacity=0.6] (-1,1) rectangle (4,4);
				
				\fill[red!20, opacity=0.5] (-0.5,-0.5) rectangle (4,4);
				
				\fill[pattern=north west lines, pattern color=gray!40] (-0.5,1) rectangle (4,4);
				
				\draw[very thin, gray!30] (-4,-4) grid (4,4);
				\draw[->, thick] (-3.5,0) -- (3.5,0); 
				\draw[->, thick] (0,-3.5) -- (0,3.5);
				
				\draw[blue, thick] (-3.5,-3.5) -- (3.5,3.5); 
				\draw[blue, thick] (-3.5,3.5) -- (3.5,-3.5); 
				
				\filldraw[black] (-1,1) circle (1.5pt);
				\node[anchor=south east, inner sep=1pt, opacity=0.8, text opacity=1] at (-1.1,0.7) {$A$};
				
				\filldraw[black] (-0.5,-0.5) circle (1.5pt);
				\node[anchor=north east, inner sep=1pt, opacity=0.8, text opacity=1] at (-0.7,-0.3) {$B$};
				\filldraw[black] (1,1) circle (1.5pt);
				\node[anchor=south east, inner sep=1pt, opacity=0.8, text opacity=1] at (1.2,0.5) {$C$};
				
				\draw[dashed, green!60!black] (-1,1) -- (-1,4);
				\draw[dashed, green!60!black] (-1,1) -- (4,1);
				
				\draw[dashed, red!60!black] (-0.5,-0.5) -- (-0.5,4);
				\draw[dashed, red!60!black] (-0.5,-0.5) -- (4,-0.5);
				
			\end{tikzpicture}
		}
		\]
		Note that 
		\[
		((-1,1) \vee (0,0)) \wedge (1,-1) = (1, 1) \wedge (1, -1) = (1,-1),
		\]
		and 
		\[
		((-1,1) \wedge (1,-1)) \vee ((0,0) \wedge (1,-1)) = (-1, -1) \vee (-1, -1) = (-1,-1).
		\]
		Therefore, $S$ is not a distributive lattice star.
	\end{example}
	
	We can define the positive part, negative part, and absolute value of any element $x \in S$ just as in a vector lattice.
	
	\begin{definition}
		Let $S$ be a lattice star. For $x \in S$, we define the \emph{positive and negative parts} of $x$ as
		\[ x^+ := x \vee 0 \quad \text{and} \quad x^- := -(x \wedge 0), \]
		and we define the \emph{absolute value} of $x$ as $\abs{x} := x \vee (-x)$.
	\end{definition}
	
	\begin{proposition}
		Let $S$ be a lattice star. Let $x, y \in S$ and $\lambda \in \bR$. Then
		\begin{enumerate}[(i)]
			\item $\abs{\lambda x} = \abs{\lambda} \abs{x}$,
			\item $\abs{x} = 0$ if and only if $x  = 0$,
			\item if $x \geq 0$, then $\abs{x} = x$, and
			\item if $x \leq 0$, then $\abs{x} = -x$.
		\end{enumerate}
		If $S$ is a distributive lattice star, we also have that
		\begin{enumerate}[(i)] \addtocounter{enumi}{4}
			\item $(x \vee y)^+ = x^+ \vee y^+$ and $(x \vee y)^- = x^- \wedge y^-$, and
			\item $(x \wedge y)^+ = x^+ \wedge y^+$ and $(x \wedge y)^- = x^- \vee y^-$.
		\end{enumerate}
	\end{proposition}
	\begin{proof}
		(i) through (iv) follow immediately from the definition of absolute value. We have that
		\[ (x \vee y)^+ = (x \vee y) \vee 0 = (x \vee 0) \vee (y \vee 0) = x^+ \vee y^+, \]
		and
		\[ (x \vee y)^- = (-(x \vee y)) \vee 0 = ((-x) \wedge (-y)) \vee 0 = ((-x) \vee 0) \wedge ((-y) \vee 0) = x^- \wedge y^-. \]
		This proves (v) and the proof of (vi) is similar.
	\end{proof}
	
	In a vector lattice, we always have that $\abs{x} \geq 0$. The following example illustrates that this need not be true in a lattice star, which motivates the need for the notion of a valuated lattice star.
	
	\begin{example}\label{E: not valuated}
		Consider $S = \overline{\mathbb{R}}$ where the nonzero elements are ordered as usual and $0$ only satisfies $-\infty < 0 < \infty$. Then $S$ is a lattice star that is not valuated since $|-1| = 1 \not\ge 0$.
	\end{example}
	
	\begin{definition}
		Let $S$ be a lattice star. We say $S$ is \emph{valuated} if for all $x \in S$, we have $\abs{x} \geq 0$.
	\end{definition}
	
	The following proposition gives several equivalent characterizations for a distributive lattice star to be valuated.
	
	\begin{proposition} \label{p:star_abs_properties}
		Let $S$ be a distributive lattice star with $x \in S$. The following are equivalent.
		\begin{enumerate}[(i)]
			\item $\abs{x} \geq 0$,
			\item $x^+ \wedge x^- = 0$, and
			\item $\abs{x} = x^+ \vee x^-$.
		\end{enumerate}
	\end{proposition}
	\begin{proof}
		We have that
		\begin{align*}
			0 \leq x^+ \wedge x^-
			& = (x \vee 0) \wedge ((-x) \vee 0) \\
			& = ((x \vee 0) \wedge (-x)) \vee ((x \vee 0) \wedge 0) \\
			& = ((x \wedge (-x)) \vee (0 \wedge (-x)) \vee ((x \vee 0) \wedge 0) \\
			& = (-\abs{x}) \vee ((-x) \wedge 0) \vee (x^+ \wedge 0).
		\end{align*}
		It follows that $x^+ \wedge x^- = 0$ if and only if $-\abs{x} \leq 0$, demonstrating the equivalence of (i) and (ii).
		Also,
		\[ x^+ \vee x^- = (x \vee 0) \vee ((-x) \vee 0) = (x \vee (-x)) \vee 0 = \abs{x} \vee 0, \]
		from which one can see that (i) and (iii) are equivalent.
	\end{proof}
	
	We also extend the notion of Dedekind completeness and the infinite distributivity property to lattice stars.
	
	\begin{definition}
		A lattice star $S$ is called \emph{Dedekind complete} if for all nonempty $A\subseteq S$ that is bounded above, we have that $\sup A$ exists in $S$. Equivalently, a lattice star is Dedekind complete if $\inf A$ exists in $S$ for every nonempty $A\subseteq S$ that is bounded below. In this case, we clearly have $\inf A=-\sup(-A)$.
	\end{definition}
	
	\begin{definition}
		A lattice star $S$ is said to have the \emph{infinite distributivity property} if whenever $A \subseteq S$ is nonempty, has a supremum in $S$, and $x \in S$, then $x \wedge A$ has a supremum in $S$ and
		\[ \sup(x \wedge A) = x \wedge \sup A. \]
		Equivalently, whenever $A$ is nonempty, has an infimum in $S$, and $x \in S$, then $x \vee A$ has an infimum in $S$ and
		\[ \inf(x \vee A) = x \vee \inf A. \]
	\end{definition}
	
	Finally, we introduce the notions of lattice substars, bands, and lattice star isomorphisms.
	
	\begin{definition}
		A subset $T$ of a lattice star $S$ is a \emph{lattice substar} of $S$ if for all $x, y \in T$ and $\lambda \in \bR$,
		we have that $\lambda x, x \vee y, x \wedge y \in T$. A lattice substar $T$ is an \emph{ideal} if whenever $x \in S$ and $y \in T$ with $\abs{x} \leq \abs{y}$, we have that $x \in T$. An ideal $T$ is a \emph{band} if whenever $A \subseteq T$ has a supremum in $S$ then $\sup A \in T$.
	\end{definition}
	
	\begin{definition}
		Let $S, T$ be lattice stars. A map $\phi \colon S \to T$ is a \emph{lattice star homomorphism} if for $x, y \in S$ and $ \lambda \in \bR$,
		\begin{enumerate}[(i)]
			\item $\phi(\lambda x) = \lambda \phi(x)$, and
			\item $\phi(x \vee y) = \phi(x) \vee \phi(y)$.
		\end{enumerate}
		If $\phi$ is bijective, then $\phi$ is called a \emph{lattice star isomorphism} and we say $S$ and $T$ are \emph{lattice star isomorphic}. Note that $\phi^{-1}$ is also a lattice star isomorphism in this case.
	\end{definition}
	
	In the next section, we will construct the sup-inf-completion of a Dedekind complete vector lattice using the sup-completion as a starting point. From the abstract point of view, the sup-completion is not a lattice substar of the sup-inf-completion because it only has positive scalar multiplication. This motivates considering lattice semistars, which are essentially stars but with only a positive scalar multiplication.
	
	\begin{definition}
		A \emph{semistar} $S$ is a set together with a positive scalar multiplication $\mathbb{R}^+\times S \to S$, $(\lambda, x) \mapsto \lambda x$ such that
		\begin{itemize}
			\item[(i)] $\lambda(\mu x)=(\lambda\mu) x\quad (\lambda,\mu\in\mathbb{R}^+, x\in S)$
			\item[(ii)] $1x=x\quad (x\in S)$
			\item[(iii)] there exists $\bar{0}\in S$ such that, for all $x\in S$, we have $0x=\bar{0}$.
		\end{itemize}
	\end{definition}
	
	From now on, the element $\overline{0}$ will simply be denoted by $0$. The notions of an ordered semistar, a lattice semistar, a valuated semistar, a distributive semistar, a Dedekind complete semistar, the infinite distributivity property for semistars, lattice semistar homomorphisms, and lattice semistar isomorphisms can be defined in the obvious ways. We leave this to the reader.
	
	\section{Construction of the sup-inf-completion}
	
	Given a Dedekind complete vector lattice $E$ and its sup-completion $E^s$, we define
	\[ E^s_+ \Delta E^s_+ := \{(x, y) \in E^s_+  \times E^s_+ \colon x \wedge y = 0 \}. \]
	We define a scalar multiplication on $E^s_+ \Delta E^s_+$ as follows. For $(x, y) \in E^s_+ \Delta E^s_+$ and $\lambda \in \bR^+$, define
	\[ \lambda(x, y) := (\lambda x, \lambda y) \text{ and } -\lambda(x, y) := (\lambda y, \lambda x). \]
	We define a partial ordering via $(x, y) \leq (z, w)$ if $x \leq z$ and $w \leq y$. It is readily verified that $E^s_+ \Delta E^s_+$ is a lattice star with
	\[ (x, y) \vee (z, w) = (x \vee z, y \wedge w). \]
	
	If $A \subseteq E^s_+ \Delta E^s_+$ is bounded above, then let $w := \sup\{x \colon (x,y) \in A\}$ and $z := \inf\{y \colon (x,y) \in A\}$. Note that by the infinite distributivity property of $E^s$ we find 
	\[
	0 \le w \wedge z = \sup\{x\wedge z \colon (x,y) \in A\} \le \sup\{ x \wedge y \colon (x,y) \in A\} = 0,
	\]
	so $(w,z) \in E^s_+ \Delta E^s_+$, and it is readily verified that $(w,z) = \sup A$. Hence $E^s_+ \Delta E^s_+$ is Dedekind complete.
	
	To see that $E^s_+ \Delta E^s_+$ is valuated, note that for $(x, y) \in E^s_+ \Delta E^s_+$,
	\[ \abs{(x, y)} = (x, y) \vee (y, x) = (x \vee y, x \wedge y) = (x \vee y, 0) \geq (0, 0). \]
	
	Furthermore, if for each $x \in E^s$, we identify $x$ with $(x^+, x^-)$, it follows that $E$ is a lattice substar of $E^s_+ \Delta E^s_+$. Recall that in the sup-completion we have that if $x \in E$ and $y \in E^s$ are such that $y \le x$, then $y \in E$. In particular, $x \wedge 0 \le 0$ implies $x \wedge 0 \in E$, so that $x^- = - (x \wedge 0) \in E$. Hence we can identify every $x \in E^s$ with $(x^+,x^-)$ , where $x^- \in E$, so that $$(x^+,x^-) = (x^+,x^-)\vee (0,x^-)$$ and the pair $(0,x^-)$ corresponds to an element of $E$. On the other hand, if $(x,y) \in  E^s_+ \Delta E^s_+$ and $(z^+, z^-)$ corresponds to an element $z \in E$, then $(x,y) \vee (z^+, z^-) = (x \vee z^+, y \wedge z^-)$, where $y \wedge z^-$ must belong to $E$, and thus $(x,y) \vee (z^+, z^-)$ corresponds to an element of $E^s$. We conclude that
	\[
	\{(x,y) \vee (z^+,z^-) \colon (x,y) \in E^s_+ \Delta E^s_+, z \in E \} \cong E^s
	\]
	as lattice semistars.
	
	The elements of $E^s_+ \Delta E^s_+$ can also be described by taking infima and suprema of disjoint elements in $E^+$, which will be a crucial part in the construction of the sup-inf-completion of $E$. Indeed, note that for any $(x,y) \in E^s_+ \Delta E^s_+$, we have
	\begin{align*}
		(x, y)
		& = \sup \{ \inf \{ (v, w) \colon w \le y,\ w \in E^+\} \colon v \le y, v \in E^+\}.
	\end{align*}
	
	In \cite{Azouzi}, Azouzi introduced the notation $[u]^\leq$ for $u \in E^s$ to refer to the set $\{x \in E \colon x \leq u \}$. We will benefit from having a similar notation when considering only positive elements. If $S$ is a lattice star containing $E$ as a lattice substar, define for $0 \leq u \in S$,
	\[ [u]^\leq_+ := \{x \in E^+ \colon x \leq u \}. \]
	
	\begin{theorem} \label{t:unique_sup_inf_completion}
		Let $E$ be a Dedekind complete vector lattice.  Then there exists an essentially unique Dedekind complete distributive valuated lattice star $S$ that contains $E$ as a lattice substar such that
		\begin{enumerate}[(i)]
			\item $T := \{s \vee x \colon s \in S, x \in E \}$ is lattice semistar isomorphic to $E^s$, and
			\item $s = \sup \{ \inf \{x - y \colon y \in [s^-]^\leq_+\} \colon x \in [s^+]^\leq_+ \}$ for all $s \in S$.
		\end{enumerate}
		
	\end{theorem}
	\begin{proof}
		It was shown above that $E^s_+ \Delta E^s_+$ satisfies these properties. To prove uniqueness, let $S$ be such a lattice star with $\psi \colon T \to E^s$ a lattice semistar isomorphism. Define $\varphi \colon S \to E^s_+ \Delta E^s_+$ by $\varphi(s) := (\psi(s^+), \psi(s^-))$. Note that $\psi$ preserves suprema, that is, $A \subseteq T$ has a supremum in $T$ if and only if $\psi(A)$ has a supremum in $E^s$---in this case, $\psi(\sup A) = \sup \psi(A)$. Furthermore, for $x, y \in E^+$, we also have that $\psi(x - y) = \psi(x) - \psi(y)$. To see this, note that $\psi(x-y)^+ = \psi((x-y)^+) = \psi(x)$ and $\psi(x - y)^- = \psi((x-y)^-) = \psi(y)$. Since $\psi$ preserves $E$ as $E^s$ is an imbedding cone of $E$ by \cite[Section 1]{Donner}, the desired equality follows.
		
		We proceed to show that $\varphi$ is a lattice star isomorphism. Let $s, t \in S$. If $s^+ = t^+$ and $s^- = t^-$, then
		\begin{align*}
			s
			& = \sup \{ \inf \{x - y \colon y \in [s^-]^\leq_+\} \colon x \in [s^+]^\leq_+ \} \\
			& = \sup \{ \inf \{x - y \colon y \in [t^-]^\leq_+\} \colon x \in [t^+]^\leq_+ \} \\
			& = t.
		\end{align*}
		Therefore, $\varphi$ is injective. Since $(s \vee t)^+ = s^+ \vee t^+$ and $(s \vee t)^- = s^- \wedge t^-$, we have that
		\begin{align*}
			\varphi(s \vee t)
			& = (\psi(s^+ \vee t^+), \psi(s^- \wedge t^-)) \\
			& = (\psi(s^+) \vee \psi(t^+), \psi(s^-) \wedge \psi(t^-)) \\
			& = (\psi(s^+), \psi(s^-)) \vee (\psi(t^+), \psi(t^-)) \\
			& = \varphi(s) \vee \varphi(t).
		\end{align*}
		For $\lambda \in \bR^+$,
		\[ \varphi(\lambda s) = (\psi(\lambda s^+), \psi(\lambda s^-)) = (\lambda \psi(s^+), \lambda \psi(s^-)) = \lambda (\psi(s^+), \psi(s^-)) = \lambda \varphi(s). \]
		Additionally,
		\[ \varphi(-s) = (\psi(s^-), \psi(s^+)) = -(\psi(s^+), \psi(s^-)) = -\varphi(s). \]
		Thus, $\varphi$ is a lattice star homomorphism.
		
		Let $(z, w) \in E^s_+ \Delta E^s_+$. There there are $s, t \in S^+$ such that $s \wedge t = 0$ and $\psi(s) = z$ and $\psi(t) = w$. Define the element
		\[ v := \sup \{ \inf \{x - y \colon y \in [t]^\leq_+ \} \colon x \in [s]^\leq_+ \}. \]
		This element is well-defined since $-t \leq x - y \leq s$ for all such $x, y$. For $x \in [s]^\leq_+$, we have that
		\begin{align*}
			\psi \left( \sup \{y - x \colon y \in [t]^\leq_+ \} \right)
			& = \sup \{\psi(y) - \psi(x) \colon y \in [t]^\leq_+ \} \\
			& = w - \psi(x),
		\end{align*}
		and hence $\sup \{y - x \colon y \in [t]^\leq_+ \} = \psi^{-1}(w - \psi(x))$. Thus,
		\[ -v = \inf \{ \sup \{y - x \colon y \in [t]^\leq_+ \} \colon x \in [s]^\leq_+ \} = \inf \{ \psi^{-1}(w - \psi(x)) \colon x \in [s]^\leq_+ \}. \]
		Now \cite[IC2, p. 3]{Donner} yields that $(w - \psi(x)) \wedge 0 = -\psi(x)$, thus by \cite[IC1, p. 3]{Donner}, we have
		\begin{align*}
			(-v) \wedge 0 
			& = \inf \{ \psi^{-1}((w - \psi(x)) \wedge 0) \colon x \in [s]^\leq_+ \} \\
			& = \inf \{-x \colon x \in [s]^\leq_+ \} \\
			& = - \sup \{x \colon x \in s^\leq_+ \} \\
			& = - s,
		\end{align*}
		so that $v^+ = s$. For any $y \in [t]^\leq_+$, we have that
		\[ v \leq \sup \{x - y \colon x \in [s]^\leq_+ \} = \psi^{-1}(z - \psi(y)), \]
		and hence $v \wedge 0 \leq \psi^{-1}(z - \psi(y)) \wedge 0 = -y$ so that $y \leq v^-$. Thus, $t \leq v^-$. Since
		\[ -v = \inf \{ \psi^{-1}(w - \psi(x)) \colon x \in [s]^\leq_+ \} \leq \psi^{-1}(w - \psi(x)) \]
		for all $x \in [s]^\leq_+$, it follows that $v^- = (-v) \vee 0 \leq \psi^{-1}((w - \psi(x)) \vee 0) = t$. Therefore, $\varphi(v) = (z, w)$, showing that $\varphi$ is a lattice star isomorphism.
	\end{proof}
	
	We call this essentially unique lattice star the \emph{sup-inf-completion} of $E$, and we denote it by $E^{si}$. Note that multiplying by $-1$ in a lattice star is an involutive order anti-isomorphism, and hence the order of taking the supremum and the infimum can be swapped, so the construction of the analogous \emph{inf-sup-completion} of $E$ would yield the same object. Thus, we have the following proposition.
	
	\begin{proposition} \label{p:inf_sup_completion}
		For any $s \in E^{si}$,
		\[ s = \inf \{ \sup \{ x - y \colon x \in [s^+]^\leq_+ \} \colon y \in [s^-]^\leq_+ \}. \]
	\end{proposition}
	
	Following the proof of \Cref{t:unique_sup_inf_completion}, we can obtain the following universal property.
	
	\begin{corollary}
		Suppose $S$ is a Dedekind complete distributive valuated lattice star containing a Dedekind complete vector lattice $E$ as a star sublattice such that
		\begin{enumerate}[(i)]
			\item there is an injective semistar homomorphism from $\{ s \vee x \colon s \in S, x \in E \}$ into $E^s$, and
			\item $s = \sup \{ \inf \{x - y \colon y \in [s^-]^\leq_+ \} \colon x \in [s^+]^\leq_+ \}$ for all $s \in S$.
		\end{enumerate}
		Then there is an injective lattice star homomorphism from $S$ into $E^{si}$.
	\end{corollary}
	
	The universal completion $E^u$ of $E$ can be represented as
	\[ E^u \cong \{f \in C(K, \overline{\bR}) \colon f^{-1}(\bR) \text{ is open and dense}\} \]
	by the Maeda-Ogasawara representation theorem \cite[Theorem 1.40]{AliprantisBurkinshaw}, where $K$ is an extremally disconnected compact Hausdorff space. It was shown in \cite[Theorem 1]{TroitskyPolavarapu} that the sup-completion of a Dedekind complete vector lattice $E$ can be identified with
	\[ E^s \cong \{f \in C(K, \overline{\bR}) \colon \text{ there is } x \in E \text{ such that } x \leq f \}, \]
	where the isomorphism is as lattice cones. Using this fact, we can now also represent the sup-inf-completion as follows.
	
	\begin{theorem}
		The sup-inf-completion $E^{si}$ of a Dedekind complete vector lattice $E$ is isomorphic to $C(K, \overline{\bR})$ as lattice stars.
	\end{theorem}
	\begin{proof}
		
		Let $\psi \colon E^s \to \{f \in C(K,\overline{\mathbb{R}}) \colon\ \mbox{there is $x \in E$ such that $x \le f$}\}$ be a lattice cone isomorphism. Define the map $\varphi \colon E^s_+ \Delta E^s_+ \to C(K,\overline{\mathbb{R}})$ by $\varphi(x,y) := \psi(x) - \psi(y)$. Since 
		\[
		\{\omega \in K \colon \psi(x)(\omega) > 0\} \cap \{\omega \in K \colon \psi(y)(\omega) > 0\} = \varnothing,
		\]
		the map is well-defined. It should be clear that $\varphi(\lambda(x,y)) = \lambda \varphi((x,y))$ for all $\lambda \in \mathbb{R}$. For $(x,y), (z,w) \in E^{si}$, it follows that 
		\begin{align*}
			\varphi((x,y) \vee (z,w)) &= \varphi((x \vee z , y \wedge w)) = \psi(x \vee z) - \psi( y \wedge w) \\&= \psi(x) \vee \psi(z) - \psi(y) \wedge \psi(w) \\&= (\psi(x) - \psi(y)) \vee (\psi(z) - \psi(w)) \\&= \varphi((x,y)) \vee \varphi((z,w)),
		\end{align*}
		hence $\varphi$ is a lattice star homomorphism. If $(x,y)$ and $(z,w)$ are such that $\psi(x) - \psi(y) = \psi(z) - \psi(w)$, then by taking the positive and negative parts, we find that $\psi(x) = \psi(z)$ and $\psi(y) = \psi(w)$. Hence $(x,y) = (z,w)$, showing that $\varphi$ is injective. Let $f \in C(K,\overline{\mathbb{R}})$. Then $f = f_+ - f_-$, since 
		\[
		\{\omega \in K \colon f(\omega) > 0\} \cap \{\omega \in K \colon f(\omega) < 0\} = \varnothing.
		\]
		Let $x, y \in E^s_+$ be such that $\psi(x) = f_+$ and $\psi(y) = f_-$. Then $\varphi((x,y)) = f$, and $\varphi$ is surjective.
		Thus, $E^{si}$ and $C(K,\overline{\mathbb{R}})$ are isomorphic as lattice stars. 
	\end{proof}
	
	We can also recover the universal completion of $E$ from the sup-inf completion.
	
	\begin{proposition}
		Let $E$ be a Dedekind complete vector lattice. Then the universal completion of $E$ is given by
		\[ E^u = E^{si}_F := \{x \in E^{si} \colon \inf_{n \in \bN} \tfrac{1}{n} \abs{x} = 0 \}. \]
	\end{proposition}
	\begin{proof}
		We will prove that
		\[ C^\infty(K) = C(K, \overline{\bR})_F = \{f \in C(K, \overline{\bR}) \colon {\textstyle \inf_{n \in \bN}} \tfrac{1}{n} \abs{f} = 0 \}. \]
		That $C^\infty(K) \subseteq C(K, \overline{\bR})_F$ follows from the fact that $C^\infty(K)$ is an Archimede-an vector lattice. Let $f \in C(K, \overline{\bR})_F$. Note that $f^{-1}(\bR)$ is open in $K$, as $f$ is continuous. Suppose $f^{-1}(\bR)$ is not dense. Then there exists a nonempty open set $O \subseteq K$ such that $O \cap f^{-1}(\bR) = \varnothing$.
		
		Since $\bar{O}$ is clopen, as $K$ is extremally disconnected, $\mathbf{1}_{\bar{O}} \in C(K)$. By continuity, $\abs{f}(t) = \infty$ for all $t \in \bar{O}$. It now follows that
		\[ 0 < \mathbf{1}_{\bar{O}} \abs{f} = \tfrac{1}{n} \mathbf{1}_{\bar{O}} \abs{f} \leq \tfrac{1}{n} \abs{f} \downarrow 0, \]
		which is absurd. Hence $f^{-1}(\bR)$ is dense in $K$ and therefore $f \in C^\infty(K)$.
	\end{proof}
	
	\section{Properties of the sup-inf-completion}
	
	This section is dedicated to generalising important notions to the sup-inf-completion, namely bands and band projections, finite and infinite parts of vectors, and various band decompositions. Throughout, $E$ is a Dedekind complete vector lattice.
	
	\begin{proposition}
		$E^{si}$ has the infinite distributivity property.
	\end{proposition}
	\begin{proof}
		The proof of part (viii) in \cite[Theorem 1]{TroitskyPolavarapu} can be followed.
	\end{proof}
	
	Since $E^s = \{ s \vee x \colon s \in E^{si}, x \in E\}$, $E^{si}$ has a largest element and hence a smallest element. Furthermore, note that if $F$ is an ideal of $E$, then $F^{si} \subseteq E^{si}$.
	
	\begin{notation}
		Given an ideal $F$ of $E$, we write $+\infty_F$ and $-\infty_F$ for the largest and smallest elements, respectively, in $F^{si}$.
	\end{notation}
	
	The following proposition describes the relations between positive and negative elements in the sup-completion (or inf-completion, which we recall is defined by $E^s$ equipped with the reversed order) and the sup-inf-completion. 
	
	\begin{proposition}
		Let $E$ be a Dedekind complete vector lattice. Then
		\begin{enumerate}[(i)]
			\item $(E^{si})^+ = (E^s)^+$, and
			\item $(E^{si})^- = (E^i)^-$.
		\end{enumerate}
	\end{proposition}
	\begin{proof}
		(i) follows immediately since $E^s = \{ s \vee x \colon s \in E^{si}, x \in E\}$, and (ii) is a consequence of the fact that $E^i = -E^s$.
	\end{proof}
	
	The absolute value on $E^s$ is defined by $\abs{x} = x^+ \vee x^-$. Since $E^{si}$ is a distributive lattice star, we also have that $\abs{x} = x^+ \vee x^-$ for all $x \in E^{si}$ by \Cref{p:star_abs_properties}(iii). Thus, the absolute value on $E^{si}$ is an extension of the absolute value on $E^s$. Next we show how to extend band projections in $E$ to the entirety of $E^{si}$. In Proposition~\ref{p:properties projections sup-inf} we prove various properties of these extensions. 
	
	\begin{remark} \label{r:projection_extensions}
		Given a band projection $\bP$ in $E$, we can extend it to $E^{si}$ by defining for any $s \in E^{si}$,
		\[ \bP(s) := \sup \{ \inf \{ \bP(x - y) \colon y \in [s^-]^\leq_+ \} \colon x \in [s^+]^\leq_+ \}. \]
		Recall that for $0 \leq u \in E^s$, we can define a band projection $\bP_u$ as in \cite[Remark 5.9]{differentiation}, i.e.  for $x \in E^+$, define
		\[ \bP_u(x) := \sup_{n \in \bN} (x \wedge n u), \]
		and then for $x \in E$, define
		\[ \bP_u(x) := \bP_u(x^+) - \bP_u(x^-). \]
		For any $s \in E^{si}$, $\abs{s} \in E^s$ and we will write $\bP_s$ for $\bP_{\abs{s}}$, and $B_s$ for $B_{\abs{s}}$. Note that if $0 \leq u, v \in E^s$, then
		\[ \bP_{u \vee v} (x) = \sup_{n \in \bN} (x \wedge n (u \vee v)) = \sup_{n \in \bN} (x \wedge nu) \vee  \sup_{n \in \bN} (x \wedge nv) = \bP_u (x) \vee \bP_v(x) \]
		for any $x \in E^+$. Thus, $B_{u \vee v} = B_u + B_v$.
	\end{remark}
	
	\begin{lemma} \label{l:supremum_order_continuous}
		Let $(s_\alpha)$ and $(t_\alpha)$ be decreasing nets in $E^{si}$. Then
		\[ \inf_\alpha (s_\alpha \vee t_\alpha) = \left( \inf_\alpha s_\alpha \right) \vee \left( \inf_\alpha t_\alpha \right). \]
	\end{lemma}
	\begin{proof}
		Let $s := \inf_\alpha s_\alpha$ and $t := \inf_\alpha t_\alpha$. It is clear that $s \vee t \leq \inf_\alpha (s_\alpha \vee t_\alpha)$. Now suppose $u \leq s_\alpha \vee t_\alpha$ for all $\alpha$.
		Fix $\beta$, then for $\alpha \geq \beta$,
		\[ u \leq s_\alpha \vee t_\alpha \leq s_\alpha \vee t_\beta. \]
		By the infinite distributivity property,
		\[ u \leq \inf_{\alpha \geq \beta} (s_\alpha \vee t_\beta) = s \vee t_\beta. \]
		Similarly, $u \leq s \vee t$. Thus, $s \vee t = \inf_\alpha (s_\alpha \vee t_\alpha)$.
	\end{proof}
	
	\begin{proposition}\label{p:properties projections sup-inf}
		Let $s, t \in E^{si}$ and let $\bP$ be a band projection in $E^{si}$.
		\begin{enumerate}[(i)]
			\item If $s \leq t$, then $\bP(s) \leq \bP(t)$,
			\item $\bP(s \vee t) = \bP(s) \vee \bP(t)$ and $\bP(s \wedge t) = \bP(s) \wedge \bP(t)$,
			\item $\bP(\lambda s) = \lambda \bP(s)$ for any $\lambda \in \bR$,
			\item if $s \geq 0$, then $\bP(s) \leq s$,
			\item $\bP(s) = \sup \{ \inf \{ x - y \colon y \in [\bP(s^-)]^\leq_+\} \colon x \in [\bP(s^+)]^\leq_+ \}$, and
			\item $s \in \bP(E^{si})$ if and only if $s = \bP(s)$.
		\end{enumerate}
	\end{proposition}
	\begin{proof}
		To prove (i), suppose $s \leq t$. Then $s^+ \leq t^+$ and $t^- \leq s^-$, and hence, for every $x \in [s^+]^\leq_+$,
		\[ \inf \{ \bP(x - y) \colon y \in [s^-]^\leq_+ \} \leq \inf \{ \bP(x - y) \colon y \in [t^-]^\leq_+ \}. \]
		Therefore,
		\begin{align*}
			\bP(s)
			& = \sup \{ \inf \{ \bP(x - y) \colon y \in [s^-]^\leq_+ \} \colon x \in [s^+]^\leq_+ \} \\
			& \leq \sup \{ \inf \{ \bP(x - y) \colon y \in [t^-]^\leq_+ \} \colon x \in [s^+]^\leq_+ \} \\
			& \leq \sup \{ \inf \{ \bP(x - y) \colon y \in [t^-]^\leq_+ \} \colon x \in [t^+]^\leq_+ \} \\
			& = \bP(t).
		\end{align*}
		We prove (ii). By (i), we have that $\bP(s) \vee \bP(t) \leq \bP(s \vee t)$. Now suppose $\bP(s) \vee \bP(t) \leq u$ for some $u \in E^{si}$. Then, for any $x \in [s^+]^\leq_+$, we have that
		\[ \inf \{ \bP(x - y) \colon y \in [s^- \wedge t^-]^\leq_+ \} \leq \inf \{ \bP(x - y) \colon y \in [s^-]^\leq_+ \} \leq u. \]
		Similarly, for any $x \in [t^+]^\leq_+$, we have that
		\[ \inf \{ \bP(x - y) \colon y \in [s^- \wedge t^-]^\leq_+ \} \leq u. \]
		Let $x \in [s^+ \vee t^+]^\leq_+$. Define $x_s := x \wedge s^+$ and $x_t := x \wedge t^+$, then $x = x_s \vee x_t$. By \Cref{l:supremum_order_continuous},
		\begin{align*}
			\inf \{ &\bP(x - y) \colon y \in [s^- \wedge t^-]^\leq_+ \}\\
			& = \inf \{ \bP(x_s - y) \vee \bP(x_t - y) \colon y \in [s^- \wedge t^-]^\leq_+ \} \\
			& = \left( \inf \{ \bP(x_s - y) \colon y \in [s^- \wedge t^-]^\leq_+ \} \right)  \vee \left( \inf \{\bP(x_t - y) \colon y \in [s^- \wedge t^-]^\leq_+ \} \right) \\
			& \leq u.
		\end{align*}
		Therefore,
		\[ \bP(s \vee t) = \sup \{ \inf \{ x - y \colon y \in [s^- \vee t^-]^\leq_+\} \colon x \in [s^+ \vee t^+]^\leq_+ \} \leq u. \]
		This proves that $\bP(s \vee t) = \bP(s) \vee \bP(t)$. That $\bP(s \wedge t) = \bP(s) \wedge \bP(t)$ will follow from (iii), which we prove next. For $ \lambda > 0$,
		\begin{align*}
			\bP(\lambda s)
			& = \sup \{ \inf \{ \bP(x - y) \colon y \in [\lambda s^-]^\leq_+ \} \colon x \in [\lambda s^+]^\leq_+ \} \\
			& = \sup \{ \inf \{ \lambda \bP(x - y) \colon y \in [s^-]^\leq_+ \} \colon x \in [s^+]^\leq_+ \} \\
			& = \lambda \sup \{ \inf \{ \bP(x - y) \colon y \in [s^-]^\leq_+ \} \colon x \in [s^+]^\leq_+ \} \\
			& = \lambda \bP(s).
		\end{align*}
		Additionally,
		\begin{align*}
			\bP(-s)
			& = \sup \{ \inf \{ \bP(x - y) \colon y \in [s^+]^\leq_+ \} \colon x \in [s^-]^\leq_+ \} \\
			& = - \inf \{ \sup \{ \bP(y - x) \colon y \in [s^+]^\leq_+ \} \colon x \in [s^-]^\leq_+ \} \\
			& = - \bP(s).
		\end{align*}
		To prove (iv), suppose $s \geq 0$. Then,
		\[ \bP(s) = \sup \{ \bP(x) \colon x \in [s]^\leq_+ \}. \]
		For every $x \in [s]^\leq_+$, $\bP(x) \leq x \leq s$ and hence $\bP(s) \leq s$. Note that by (ii) $\bP(s)^+ = \bP(s) \vee 0 = \bP(s \vee 0) = \bP(s^+)$ and, similarly, $\bP(s)^- = \bP(s^-)$. Then (v) follows from \Cref{t:unique_sup_inf_completion}(ii). From (v), it follows that $\bP(\bP(s)) = \bP(s)$ since we know $\bP^2 = \bP$ on $E^s$. Now, if $s = \bP(t)$ for some $t \in \bP(E^{si})$, then
		\[ s = \bP(t) = \bP(\bP(t)) = \bP(s), \]
		which proves (vi).
	\end{proof}
	
	\begin{proposition}
		The bands in $E^{si}$ are precisely the subsets $\bP(E^{si})$ where $\bP$ is a band projection in $E$ extended to $E^{si}$.
	\end{proposition}
	\begin{proof}
		Let $\bP$ be a band projection in $E$ extended to $E^{si}$. Consider $s \in E^{si}$. Then $s^+, s^- \in \bP(E^s)$ if and only if $\abs{s} \in \bP(E^s)$ by \Cref{p:projection_image_ideal}. By \Cref{p:properties projections sup-inf}(ii, iii, vi), it is clear that $\bP(E^{si})$ is a lattice substar of $E^{si}$ and that $s \in \bP(E^{si})$ implies $s^+, s^- \in E^{si}$. Conversely, if $s^+, s^- \in \bP(E^s)$, then
		\begin{align*}
			s
			& = \sup \{ \inf \{ x - y \colon y \in [s^-]^\leq_+ \} \colon x \in [s^+]^\leq_+ \} \\
			& = \sup \{ \inf \{ x - y \colon y \in [\bP(s^-)]^\leq_+ \} \colon x \in [\bP(s^+)]^\leq_+ \} \\
			& = \bP(s).
		\end{align*}
		Therefore, $s \in \bP(E^{si})$ if and only if $\abs{s} \in \bP(E^s)$. If $t \in E^{si}$ with $\abs{t} \leq \abs{s}$, then $\abs{t} \in \bP(E^s)$ since $\bP(E^s)$ is an ideal in $E^s$, and hence $t \in \bP(E^{si})$.
		
		Let $A \subseteq \bP(E^{si})$ be bounded above. Define $A' = A \vee x$ for some element $x \in A$. Then $A'$ is bounded above and below, with $\sup(A) = \sup(A')$. Now $\abs{A'} := \{\abs{s} \colon s \in A'\}$ is bounded above, and hence $\sup (\abs{A'}) \in \bP(E^s)$. Since $\abs{\sup(A)} = \abs{\sup(A')} \leq \sup(\abs{A'})$, $\sup(A) \in \bP(E^{si})$. This concludes the proof that $\bP(E^{si})$ is a band in $E^{si}$.
		
		Let $B$ be an arbitrary band in $E^{si}$, and let $u := \sup B$. Since $0 \in B$, $0 \leq u \in E^s$. Note also that $nu = u$ for all $n \in \bN$. We will show that $B = \bP_u(E^{si})$. For any $x \in E^+$,
		\[ \bP_u(x) = \sup_{n \in \bN} \{x \wedge n u \} = x \wedge u. \]
		Let $s \in B$. Then, for any $x \in [s^+]^\leq_+$ and $y \in [s^-]^\leq_+$, we have that $x, y \in B$ and so $\bP_u(x) = x \wedge u = x$ and $\bP_u(y) = y \wedge u = y$. Therefore,
		\begin{align*}
			\bP_u(s)
			& = \sup \{ \inf \{ \bP_u(x - y) \colon y \in [s^-]^\leq_+ \} \colon x \in [s^+]^\leq_+ \} \\
			& = \sup \{ \inf \{ x - y \colon y \in [s^-]^\leq_+ \} \colon x \in [s^+]^\leq_+ \} \\
			& = s.
		\end{align*}
		Thus, $B \subseteq \bP_u(E^{si})$. Let $s = \bP_u(t)$ for some $t \in E^{si}$. Now,
		\[ s = \sup \{ \inf \{ \bP_u(x - y) \colon y \in [t^-]^\leq_+ \} \colon x \in [t^+]^\leq_+ \}. \]
		For such $x, y$, we have that $-u \leq \bP_u(x - y) \leq u$ so $\bP_u(x - y) \in B$ and hence $s \in B$.
	\end{proof}
	
	\begin{remark}
		There is a one-to-one correspondence between bands in $E$ and $E^{si}$, but also with bands in $E^{u}$, $E^s$, or $E^i$. We will always write $B$ for the band in $E$, $B^{si}$ for the band in $E^{si}$, $B^u$ for the band in $E^u$, etc. On the other hand, we will write $\bP$ for the band projection of $B^{si}$ as this is the largest of the extensions of the band projection of $B$. \end{remark}
	
	The following two lemmas will prove useful when dealing with sums of bands.
	
	\begin{lemma} \label{l:adding_projections}
		Let $0 \leq s \in E^s$. Further, let $B, B_1$, and $B_2$ be bands such that $B = B_1 + B_2$, then $\bP(s) = \bP_1(s) \vee \bP_2(s)$.
	\end{lemma}
	\begin{proof}
		We have
		\begin{align*}
			\bP(s)
			& = \sup \{ \inf \{ \bP(x - y) \colon y \in [s^-]^\leq_+ \} \colon x \in [s^+]^\leq_+ \} \\
			& = \sup \{ \bP(x) \colon x \in [s]^\leq_+ \} \\
			& = \sup \{ \bP_1(x) \vee \bP_2(x) \colon x \in [s]^\leq_+\} \\
			& = \left( \sup \{ \bP_1(x) \colon x \in [s]^\leq_+ \} \right) \vee \left( \sup \{ \bP_2(x) \colon x \in [s]^\leq_+ \} \right) \\
			& = \bP_1(s) \vee \bP_2(s). \qedhere
		\end{align*}
	\end{proof}
	
	\begin{lemma} \label{l:adding_bands_inequalities}
		Let $s, t \in E^{si}$. If $B = B_1 + B_2$ as bands, if $\bP_1(s) \leq \bP_1(t)$, and if $\bP_2(s) \leq \bP_2(t)$, then $\bP(s) \leq \bP(t)$.
	\end{lemma}
	\begin{proof}
		Note that $\bP(s^+) = \bP_1(s^+) \vee \bP_2(s^+) \leq \bP_1(t^+) \vee \bP_2(t^+) = \bP(t^+)$ and $\bP(s^-) = \bP_1(s^-) \vee \bP_2(s^-) \geq \bP_1(t^-) \vee \bP_2(t^-) = \bP(t^-)$ by \Cref{l:adding_projections}. Therefore,
		\begin{align*}
			\bP(s)
			& = \sup \{ \inf \{ x - y \colon y \in [\bP(s^-)]^\leq_+ \} \colon x \in [\bP(s^+)]^\leq_+ \} \\
			& \leq \sup \{ \inf \{ x - y \colon y \in [\bP(t^-)]^\leq_+ \} \colon x \in [\bP(t^+)]^\leq_+ \} \\
			& = \bP(t). \qedhere
		\end{align*}
	\end{proof}
	
	Although we lack a general notion of addition in $E^{si}$, we can add bands and band projections in $E^{si}$, as this corresponds to the addition of bands in $E$. Thus we can prove the analogous statement $B_s = B_{s^+} + B_{s^-}$.  
	
	\begin{proposition}
		Let $s \in E^{si}$. Then
		\[ B_s = B_{s^+} \oplus B_{s^-}. \]
	\end{proposition}
	\begin{proof}
		That $B_{s^+} \cap B_{s^-} = \{0\}$ follows from the fact that $s^+ \wedge s^- = 0$. Let $x \in E^+$. Then
		\begin{align*} \bP_s(x) &= \sup_{n \in \bN} \left( x \wedge n \abs{s} \right) = \sup_{n \in \bN} \left( x \wedge n (s^+ \vee s^-) \right) \\&= \sup_{n \in \bN} \left( (x \wedge n s^+) \vee (x \wedge n s^-) \right). \end{align*}
		It follows that 
		\[\bP_{s^+}(x), \bP_{s^-}(x) \leq \bP_s(x) \leq \bP_{s^+} (x) \vee \bP_{s^-} (x) = \bP_{s^+} (x) + \bP_{s^-} (x).
		\]From the first inequality, we see that $B_{s^+}, B_{s^-} \subseteq B_s$ hence $B_{s^+} + B_{s^-} \subseteq B_s$, and from the second inequality, it follows that $B_s \subseteq B_{s^+} + B_{s^-}$. Therefore, $B_s = B_{s^+} + B_{s^-}$.
	\end{proof}
	
	As in the sup-completion, we are able to split an element of $E^{si}$ into its finite and infinite parts. Additionally, we now also have positive infinite and negative infinite parts of elements in $E^{si}$.
	
	\begin{definition}
		Let $s \in E^{si}$. We define the \emph{infinite}, \emph{positive infinite}, \emph{negative infinite}, and \emph{finite parts} of $s$, respectively, by
		\begin{enumerate}[(i)]
			\item $s_{\pm \infty} := \sup_{m \in \bN} \tfrac{1}{m} \inf_{n \in \bN} \tfrac{1}{n} s$,
			\item $s_{+ \infty} := (s_{\pm \infty})^+$,
			\item $s_{- \infty} := - (s_{\pm \infty})^-$, and
			\item $s_\cF := \sup \{ \inf \{x - x \wedge s_{+\infty} - (y - y \wedge(- s_{-\infty})) \colon y \in [s^-]^\leq_+ \} \colon x \in [s^+]^\leq_+ \}$.
		\end{enumerate}
	\end{definition}
	
	\begin{proposition} \label{p:finite_infinite_parts_formulas}
		The following identities hold for $s \in E^{si}$.
		\begin{enumerate}[(i)]
			\item $s_{+\infty} = \inf_{n \in \bN} \tfrac{1}{n} s^+$,
			\item $s_{-\infty} = - \inf_{n \in \bN} \tfrac{1}{n} s^-$, and
			\item $\abs{s_{\pm \infty}} = \inf_{n \in \bN} \tfrac{1}{n} \abs{s} = \abs{s}_{+\infty}$
		\end{enumerate}
	\end{proposition}
	\begin{proof}
		To prove (i), note that, by the infinite distributivity property,
		\begin{align*}
			s_{+\infty}
			& = \left( \sup_{m \in \bN} \tfrac{1}{m} \inf_{n \in \bN} \tfrac{1}{n} s \right) \vee 0 \\
			& = \sup_{m \in \bN} \tfrac{1}{m} \inf_{n \in \bN} \tfrac{1}{n} (s \vee 0) \\
			& = \inf_{n \in \bN} \tfrac{1}{n} s^+.
		\end{align*}
		The proof of (ii) is similar, and then (iii) follows by using \Cref{l:supremum_order_continuous} since
		\begin{align*}
			\abs{s_{\pm \infty}}
			& = (s_{+\infty}) \vee (-s_{-\infty}) \\
			& = \left( \inf_{n \in \bN} \tfrac{1}{n} s^+ \right) \vee \left( \inf_{n \in \bN} \tfrac{1}{n} s^- \right) \\
			& = \inf_{n \in \bN} \tfrac{1}{n} (s^+ \vee s^-) \\
			& = \inf_{n \in \bN} \tfrac{1}{n} \abs{s}.\qedhere
		\end{align*}
	\end{proof}
	
	The following proposition shows in what sense an element $s \in E^{si}$ is decomposable into its finite and infinite parts.
	
	\begin{proposition}
		Let $s \in E^{si}$. Then
		\begin{enumerate}[(i)]
			\item $B_s = B_{s_{\pm \infty}} \oplus B_{s_{\cF}}$,
			\item $\bP_{s_{\pm \infty}}(s) = s_{\pm \infty}$ and $\bP_{s_\cF}(s) = s_\cF$,
			\item $\abs{s_{\pm \infty}}$ is the largest element of $B_{s_{\pm \infty}}^{si}$,
			\item $s_{+\infty}$ is the largest element of $B_{s_{+\infty}}^{si}$, and
			\item $-s_{-\infty}$ is the largest element of $B_{s_{-\infty}}^{si}$.
		\end{enumerate}
	\end{proposition}
	\begin{proof}
		Let $t := \abs{s}$. Then $t_{+\infty} = \abs{s_{\pm \infty}}$ and $t_\cF = \abs{s_\cF}$ by \Cref{p:finite_infinite_parts_formulas}, so that $B_s = B_t$, $B_{s_{\pm \infty}} = B_{t_{\pm \infty}}$ and $B_{s_\cF} = B_{t_\cF}$. Now (i) follows from \cite[Theorem~5.14]{differentiation}.
		
		Let $u := s_{\pm \infty}$. It follows from the definition of $s_{\pm \infty}$ that $nu = u$ for all $n \in \bN$. Then, using \Cref{p:finite_infinite_parts_formulas},
		\begin{align*}
			\bP_u(s^+)
			& = \sup \{ \bP_u(x) \colon x \in [s^+]^\leq_+ \} = \sup \{ x \wedge \abs{u} \colon x \in [s^+]^\leq_+ \}  = s^+ \wedge \abs{u} \\
			& = s^+ \wedge \left( \inf_{n \in \bN} \tfrac{1}{n} \abs{s} \right)  = \inf_{n \in \bN} ((\tfrac{1}{n} \abs{s}) \wedge s^+) = \inf_{n \in \bN} (\tfrac{1}{n}(s^+ \vee s^-) \wedge s^+) \\
			& = \inf_{n \in \bN} ((\tfrac{1}{n}s^+ \wedge s^+) \vee (\tfrac{1}{n} s^- \wedge s^+))  = \inf_{n \in \bN} \tfrac{1}{n} s^+ \\
			& = s_{+ \infty}.
		\end{align*}
		Similarly, $\bP_u(s^-) = -s_{-\infty}$. Therefore,
		\begin{align*}
			\bP_u(s)
			& = \sup \{ \inf \{ x - y \colon y \in [\bP_u(s^-)]^\leq_+ \} \colon x \in [\bP_u(s^+)]^\leq_+ \} \\
			& = \sup \{ \inf \{ x - y \colon y \in [- s_{-\infty}]^\leq_+ \} \colon x \in [s_{+\infty}]^\leq_+ \} \\
			& = s_{\pm \infty}.
		\end{align*}
		
		To prove (iii), let $w := \abs{s_{\pm \infty}}$. Then $nw = w$ for all $n \in \bN$, and, for every $x \in E^+$, $\bP_u(x) = \sup_{n \in \bN} (x \wedge n w) = x \wedge w$. It follows that $w$ is the largest element of $B_u$. (iv) and (v) follow from (iii).
	\end{proof}
	
	In a vector lattice, we write $x \ll y$ to mean $y - x$ is a weak order unit. We extend this notation to the sup-inf-completion, but the definition must be altered to account for the fact that we do not have addition defined on all of $E^{si}$.
	
	\begin{definition} \label{d:ll}
		Let $s, t \in E^{si}$, and let $B$ be a band in $E$ with band projection $\bP$. We write $s \ll_B t$ to mean $\bP(s_{+\infty}) = \bP(t_{_-\infty}) = 0$ and $\bP \bP_{s_{-\infty}}^d \bP_{t_{+\infty}}^d (t_\cF - s_\cF)$ is a weak order unit of $(B \cap B_{s_{-\infty}}^d \cap B_{t_{+\infty}}^d)^u$. We write $s \ll t$ instead of $s \ll_E t$.
	\end{definition}
	
	Note that if $s, t \in E$, then the definition of $s \ll_B t$ reduces to simply being $\bP(t - s)$ is a weak order unit of $B$. We provide some basic properties of the relation $\ll$ in $E^{si}$.
	
	\begin{proposition} \label{p:properties_of_ll}
		Let $s, t, u \in E^{si}$, and suppose $B_1$ and $B_2$ are bands in $E$ with $B_1 \subseteq B_2$. The following holds.
		\begin{enumerate}[(i)]
			\item If $s \ll t$, then $s \leq t$,
			\item if $s \ll t$ and $u \leq s$, then $u \ll t$,
			\item if $s \ll t$ and $t \leq u$, then $s \ll u$,
			\item if $s \ll t$ and $u \ll t$, then $s \vee u \ll t$, and
			\item if $s \ll_{B_2} t$, then $s \ll_{B_1} t$.
		\end{enumerate}
	\end{proposition}
	\begin{proof}
		Suppose $s \ll t$. Then $s_{+\infty} = t_{-\infty} = 0$ and $\bP_{s_{-\infty}}^d \bP_{t_{+\infty}}^d (t_\cF - s_\cF)$ is a weak order unit of $(B_{s_{-\infty}}^d)^u \cap (B_{t_{+\infty}}^d)^u$. To prove (i), observe that
		\[ \bP_{s_{-\infty}}(s) = s_{-\infty} \leq \bP_{s_{-\infty}}(t), \]
		and
		\[ \bP_{s_{-\infty}}^d \bP_{t_{+\infty}}^d(s) = \bP_{s_{-\infty}}^d \bP_{t_{+\infty}}^d (s_\cF) \leq \bP_{s_{-\infty}}^d \bP_{t_{+\infty}}^d (t_\cF) = \bP_{s_{-\infty}}^d \bP_{t_{+\infty}}^d (t), \]
		and also
		\[ \bP_{s_{-\infty}}^d \bP_{t_{+\infty}} (s) \leq \bP_{s_{-\infty}}^d \bP_{t_{+\infty}} (t_{+\infty}) = \bP_{s_{-\infty}}^d \bP_{t_{+\infty}} (t). \]
		Since $E = B_{s_{-\infty}} \oplus B_{s_{-\infty}}^d \cap B_{t_{+\infty}}^d \oplus B_{s_{-\infty}}^d \cap B_{t_{+\infty}}$, it follows from \Cref{l:adding_bands_inequalities} that $s \leq t$. This completes the proof of (i). If $u \leq s$, then $0 \leq u_{+\infty} \leq s_{+\infty} = 0$,
		\[ \bP_{u_{-\infty}}^d (u_\cF) = \bP_{u_{-\infty}}^d (u) \leq \bP_{u_{-\infty}}^d (s) = \bP_{u_{-\infty}}^d (s_\cF) \]
		where the final equality holds because $B_{u_{-\infty}}^d \subseteq B_{s_{-\infty}}^d$, and hence since any element dominating a weak order unit is itself a weak order unit and
		\[ \bP_{u_{-\infty}}^d \bP_{t_{+\infty}}^d (t_\cF - u_\cF) \geq \bP_{u_{-\infty}}^d \bP_{t_{+\infty}}^d (t_\cF - s_\cF), \]
		we have that $\bP_{u_{-\infty}}^d \bP_{t_{+\infty}}^d (t_\cF - u_\cF)$ is a weak order unit of $(B_{u_{-\infty}}^d \cap B_{t_{+\infty}}^d)^u$. This proves (ii), and the proof of (iii) is similar.
		Suppose that $u \ll t$. Then $u_{+\infty} = 0$ and $\bP_{u_{-\infty}}^d \bP_{t_{+\infty}}^d (t_\cF - u_\cF)$ is a weak order unit of $(B_{u_{-\infty}}^d \cap B_{t_{+\infty}}^d)^u$. Note that $(s \vee u)_{+\infty} = s_{+\infty} \vee u_{+\infty} = 0$ by the infinite distributivity property. Since $(s \vee u)_{-\infty} = s_{-\infty} \vee u_{-\infty} \leq 0$, we have that $B_{(s \vee u)_{-\infty}} = B_{s_{-\infty}} \cap B_{u_{-\infty}}$ by \Cref{r:projection_extensions}, and hence $B_{(s \vee u)_{-\infty}}^d = B_{s_{-\infty}}^d + B_{u_{-\infty}}^d$.
		
		Since $(s \vee u)_{+\infty} = 0$, we have that $s \vee u \in (E^u)^i$. Now,
		\[\bP_{s_{-\infty}}^d(s \vee u) \geq \bP_{s_{-\infty}}^d(s)  = s_\cF, \]
		and hence $\bP_{s_{-\infty}}^d(s \vee u) \in E^u$. Thus,
		\[ \bP_{s_{-\infty}}^d((s \vee u)_\cF) = \bP_{s_{-\infty}}^d((s \vee u)) \geq \bP_{s_{-\infty}}^d (s_\cF), \]
		and hence $\bP_{s_{-\infty}}^d \bP_{t_{+\infty}}^d (t_\cF - (s \vee u)_\cF) $ is a weak order unit of $(B_{s_{-\infty}}^d \cap B_{t_{+\infty}}^d)^u$. Similarly, $\bP_{u_{-\infty}}^d \bP_{t_{+\infty}}^d (t_\cF - (s \vee u)_\cF) $ is a weak order unit of $(B_{u_{-\infty}}^d \cap B_{t_{+\infty}}^d)^u$. Therefore,  by \cite[Theorem 32.5]{babyzaanen}, $\bP_{(s \vee u)_{-\infty}}^d \bP_{t_{+\infty}}^d (t_\cF - (s \vee u)_\cF) $ is a weak order unit of $(B_{(s \vee u)_{-\infty}}^d \cap B_{t_{+\infty}}^d)^u$. We have shown that $s \vee u \ll t$, completing the proof of (iv). Statement (v) is an immediate consequence of \Cref{d:ll}.
	\end{proof}
	
	The definition of a weak order unit can be extended to elements of $E^{si}$. For $s \in E$, $0 \ll s$ means that $s$ is a weak order unit by definition, and in \Cref{p:wou_gg_zero} we show that this is still true for $s \in E^{si}$.
	
	\begin{definition}
		Let $s \in E^{si}$, and let $B$ be a band in $E$. We call $s$ a \emph{weak order unit} of $B$ if $s \geq 0$ and $B_s = B$. 
	\end{definition}
	
	\begin{proposition} \label{p:wou_gg_zero}
		Let $s \in E^{si}$. Then $s$ is a weak order unit of $E$ if and only if $s \gg 0$.
	\end{proposition}
	\begin{proof}
		By definition, $s \gg 0$ if and only if $s_{-\infty} = 0$ and $\bP_{s_{+\infty}}^d (s_\cF)$ is a weak order unit of $B_{s_{+\infty}}^d$. However, $\bP_{s_{+\infty}}^d (s_\cF)$ being a weak order unit of $B_{s_{+\infty}}^d$ implies that $\bP_{s_{-\infty}}(s_\cF)$ is a weak order unit of $B_{s_{-\infty}} \subseteq B_{s_{+\infty}}^d$ and hence $s_{-\infty} = 0$ since $\bP_{s_{-\infty}}(s_\cF) = 0$. Thus, we have the equivalence $s \gg 0$ if and only if $\bP_{s_{+\infty}}^d (s_\cF)$ is a weak order unit of $B_{s_{+\infty}}^d$. Since $s_{+\infty}$ is a weak order unit of $B_{s_{+\infty}}$, we have that $s \gg 0$ is also equivalent to $s$ being a weak order unit of $E$.
	\end{proof}
	
	For $x,y \in E$, we have the decomposition $E = B_{x < y} \oplus B_{x > y} \oplus B_{x = y}$. We define the necessary bands in order to obtain analogous band decompositions for elements of $E^{si}$. 
	
	\begin{definition}
		Let $s, t \in E^{si}$. We define
		\begin{align*}
			B_{s \leq t} & := B_{t_{+ \infty}} + B_{s_{- \infty}} + B_{s_{\pm \infty}}^d \cap B_{t_{\pm \infty}}^d \cap B_{s_\cF \leq t_\cF}, \\
			B_{s = t} & := B_{s \leq t} \cap B_{s \geq t}, \text{ and} \\
			B_{s < t} & := B_{t_{+\infty}} \cap B_{s_{+\infty}}^d + B_{s_{-\infty}} \cap B_{t_{-\infty}}^d + B_{t_{\pm \infty}}^d \cap B_{s_{\pm \infty}}^d \cap B_{s_{\cF} < t_{\cF}}.
		\end{align*}
	\end{definition}
	
	\begin{proposition}
		Let $s, t \in E^{si}$. Then
		\[ E = B_{s \leq t} \oplus B_{t < s} = B_{s = t} \oplus B_{s < t} \oplus B_{t < s}. \]
	\end{proposition}
	\begin{proof}
		We start by proving the first equality. Let $B := B_{s \leq t} + B_{t < s}$. Note that $B_{t_{+\infty}}, B_{s_{-\infty}} \subseteq B_{s \leq t} \subseteq B$. Furthermore, $B_{s_{+\infty}} \cap B_{t_{+ \infty}}^d \subseteq B_{t < s}$, so that
		\[ B_{s_{+\infty}} = B_{s_{+\infty}} \cap B_{t_{+\infty}} + B_{s_{+\infty}} \cap B_{t_{+ \infty}}^d \subseteq B_{t_{+\infty}} + B_{s_{+\infty}} \cap B_{t_{+ \infty}}^d \subseteq B. \]
		Similarly, $B_{t_{-\infty}} \subseteq B$. Hence, $B_{s_{\pm \infty}}, B_{t_{\pm \infty}} \subseteq B$. Finally, we have that
		\begin{align*}
			B_{s_{\pm \infty}}^d \cap B_{t_{\pm \infty}}^d & = \left(B_{s_{\pm \infty}}^d \cap B_{t_{\pm \infty}}^d\right) \cap (B_{s_\cF \leq t_{\cF}} \oplus B_{t_{\cF} < s_{\cF}}) \\& = B_{s_{\pm \infty}}^d \cap B_{t_{\pm \infty}}^d \cap B_{s_\cF \leq t_\cF} + B_{t_{\pm \infty}}^d \cap B_{s_{\pm \infty}}^d \cap B_{t_{\cF} < s_{\cF}} \\ &\subseteq B.
		\end{align*}
		
		Since the projections bands in $E$ form a Boolean algebra, we have that $B_{s_{\pm \infty}}^d \cap B_{t_{\pm \infty}}^d = (B_{s_{\pm \infty}} + B_{t_{\pm \infty}})^d$. Hence, it follows that we have the inclusion 
		\[
		E = B_{s_{\pm \infty}} + B_{t_{\pm \infty}} + B_{s_{\pm \infty}}^d \cap B_{t_{\pm \infty}}^d \subseteq B. 
		\]
		It is easily verified that $B_{s \leq t} \cap B_{t < s} = \{0\}$ by intersecting the relevant terms in the definitions.
		
		It is clear that $B_{s < t}$, $B_{s = t} \subseteq B_{s \leq t}$. Note that by distributing the intersections over the sums, we obtain
		\[ B_{s = t} = B_{s_{+\infty}} \cap B_{t_{+\infty}} + B_{s_{-\infty}} \cap B_{t_{-\infty}} + B_{s_{\pm \infty}}^d \cap B_{t_{\pm \infty}}^d \cap B_{s_\cF = t_\cF}. \]
		Furthermore, by distributing the intersection over the sum, we have that 
		\[
		B_{t_{+\infty}} = \left(B_{s_{+\infty}} + B_{s_{+\infty}}^d\right) \cap B_{t_{+\infty}} = B_{s_{+\infty}} \cap B_{t_{+\infty}} + B_{t_{+\infty}} \cap B_{s_{+\infty}}^d \subseteq B_{s = t} + B_{s < t}.
		\]
		Similarly, $B_{s_{-\infty}} \subseteq B_{s = t} + B_{s < t}$. Finally, we find that 
		\begin{align*}
			B_{s_{\pm \infty}}^d \cap B_{t_{\pm \infty}}^d \cap B_{s_\cF \leq t_\cF}
			& = B_{s_{\pm \infty}}^d \cap B_{t_{\pm \infty}}^d \cap B_{s_\cF = t_\cF} + B_{s_{\pm \infty}}^d \cap B_{t_{\pm \infty}}^d \cap B_{s_\cF < t_\cF}  \\
			& \subseteq B_{s = t} + B_{s < t}.
		\end{align*}
		Hence $B_{s \leq t} \subseteq B_{s = t} + B_{s < t}$. It is again easily verified that $B_{s < t} \cap B_{s = t} = \{0\}$ by intersecting the relevant terms in the definitions.
	\end{proof}
	
	The following proposition shows that we maintain the characterisations of the bands $B_{s \leq t}$, $B_{s = t}$, and $B_{s < t}$ that we had in $E$. 
	
	\begin{proposition} \label{p:characterize_inequality_bands}
		Let $s, t \in E^{si}$. Then
		\begin{enumerate}[(i)]
			\item $B_{s \leq t}$ is the largest band $B$ such that $\bP(s) \leq \bP(t)$,
			\item $B_{s = t}$ is the largest band $B$ such that $\bP(s) = \bP(t)$, and
			\item $B_{s < t}$ is the largest band $B$ such that $s \ll_B t$. 
		\end{enumerate}
	\end{proposition}
	\begin{proof}
		(i) That $B_{s \leq t}$ satisfies this condition follows by \Cref{l:adding_bands_inequalities} since
		\[ \bP_{t_{+\infty}}(s) \leq \infty_{B_{t_{+\infty}}} = \bP_{t_{+\infty}}(t), \]
		\[ \bP_{s_{-\infty}}(s) = -\infty_{B_{s_{-\infty}}} \leq \bP_{s_{-\infty}}(t), \]
		and for $\bP := \bP_{s_{\pm \infty}}^d \bP_{t_{\pm \infty}}^d \bP_{s_\cF \leq t_\cF}$,
		\[ \bP(s) = \bP(s_\cF) \leq \bP(t_\cF) = \bP(t). \]
		
		Suppose $B$ is a band such that $\bP(s) \leq \bP(t)$. Then the following two inclusions $B \cap B_{s_{+\infty}} \subseteq B \cap B_{t_{+\infty}}$ and $B \cap B_{t_{-\infty}} \subseteq B \cap B_{s_{-\infty}}$ hold. We can decompose $B$ as follows
		\begin{align*}
			B = B \cap (B_{s_{+\infty}} + B_{s_{-\infty}} + B_{s_{\pm \infty}}^d) \cap (B_{t_{+\infty}} + B_{t_{-\infty}} + B_{t_{\pm \infty}}^d).
		\end{align*}
		We have that
		\[ B \cap B_{t_{+\infty}} \cap (B_{s_{+\infty}} + B_{s_{-\infty}} + B_{s_{\pm \infty}}^d) \subseteq B_{t_{+\infty}} \subseteq B_{s \leq t}, \]
		and
		\[ B \cap B_{t_{-\infty}} \cap (B_{s_{+\infty}} + B_{s_{-\infty}} + B_{s_{\pm \infty}}^d) \subseteq B \cap B_{s_{-\infty}} \subseteq B_{s_{-\infty}} \subseteq B_{s \leq t}. \]
		Similarly, 
		\[ B \cap B_{s_{+\infty}} \cap (B_{t_{+\infty}} + B_{t_{-\infty}} + B_{t_{\pm \infty}}^d) \subseteq B_{s \leq t}, \]
		and 
		\[ B \cap B_{s_{-\infty}} \cap (B_{t_{+\infty}} + B_{t_{-\infty}} + B_{t_{\pm \infty}}^d) \subseteq B_{s \leq t}. \]
		Note that we only need to verify $B' := B \cap B_{s_{\pm \infty}}^d \cap B_{t_{\pm \infty}}^d \subseteq B_{s \leq t}$, and this inclusion follows immediately from the fact that
		\[ \bP'(s_\cF) = \bP'(s) \leq \bP'(t) = \bP'(t_\cF), \]
		since this implies $B' \subseteq B_{s_{\pm \infty}}^d \cap B_{t_{\pm \infty}}^d \cap B_{s_{\cF} \leq t_{\cF}} \subseteq B_{s \leq t}$. Therefore, $B \subseteq B_{s \leq t}$. This proves (i), and (ii) follows immediately from (i).
		
		We now prove (iii). By \Cref{p:inequality_decomposition}(iii), $B_{s_\cF < t_\cF}$ is the largest band $B$ such that $\bP(t_F - s_F)$ is a weak order unit of $B_{s_\cF < t_\cF}^u$. Let $B'$ be a band with corresponding band projection $\bP'$ such that $\bP_{t_{\pm \infty}}^d \bP_{s_{\pm \infty}}^d \bP' (t_\cF - s_\cF)$ is a weak order unit of $B_{t_{\pm \infty}}^d \cap B_{s_{\pm \infty}}^d \cap B'$. It follows from 
		\begin{align*}
			B'&= (B_{t_{\pm \infty}}^d \cap B_{s_{\pm \infty}}^d + (B_{t_{\pm \infty}}^d \cap B_{s_{\pm \infty}}^d)^d) \cap B' \\&= (B_{t_{\pm \infty}}^d \cap B_{s_{\pm \infty}}^d + (B_{t_{\pm \infty}} + B_{s_{\pm \infty}})) \cap B' \\& = B_{t_{\pm \infty}}^d \cap B_{s_{\pm \infty}}^d \cap B' + (B_{t_{\pm \infty}} + B_{s_{\pm \infty}}) \cap B',
		\end{align*}
		hence $$B'\subseteq B_{s_\cF < t_\cF} + B_{t_{\pm \infty}} + B_{s_{\pm \infty}}.$$  
		Since the band $B_{s_\cF < t_\cF} + B_{t_{\pm \infty}} + B_{s_{\pm \infty}}$ with corresponding band projection $\bQ$ has the property that $\bP_{t_{\pm \infty}}^d \bP_{s_{\pm \infty}}^d \bQ (t_\cF - s_\cF)$ is a weak order unit of $$B_{t_{\pm \infty}}^d \cap B_{s_{\pm \infty}}^d \cap (B_{s_\cF < t_\cF} + B_{t_{\pm \infty}} + B_{s_{\pm \infty}}),$$
		it follows that $B_{t_{\pm \infty}} + B_{s_{\pm \infty}} + B_{s_\cF < t_\cF}$ is the largest band $B'$ such that $\bP_{t_{\pm \infty}}^d \bP_{s_{\pm \infty}}^d \bP' (t_\cF - s_\cF)$ is a weak order unit of $B_{t_{\pm \infty}}^d \cap B_{s_{\pm \infty}}^d \cap B'$. Consequently, the largest band $B'$ that satisfies the conditions of $\bP'(s_{+\infty}) = \bP'(t_{-\infty}) = 0$ and $\bP_{t_{\pm \infty}}^d \bP_{s_{ \pm \infty}}^d \bP'(t_\cF - s_\cF)$ being a weak order unit of the band given by $(B_{t_{\pm \infty}}^u)^d \cap (B_{s_{ \pm \infty}}^u)^d \cap (B')^u$ is
		\begin{align*}
			(B_{t_{\pm \infty}} + &B_{s_{\pm \infty}} + B_{s_\cF < t_\cF}) \cap B_{s_{+\infty}}^d \cap B_{t_{-\infty}}^d \\
			& = B_{t_{+\infty}} \cap B_{s_{+\infty}}^d + B_{s_{-\infty}} \cap B_{t_{-\infty}}^d + B_{s_{+\infty}}^d \cap B_{t_{-\infty}}^d \cap B_{s_\cF < t_\cF}.
		\end{align*}
		These conditions together are equivalent to saying $s \ll_{B'} t$. Note that since $B^d_{s_{\pm\infty}} \cap B^d_{t_{\pm\infty}} \subseteq B^d_{s_{+\infty}} \cap B^d_{t_{-\infty}}$,
		\begin{align*}
			B^d_{s_{+\infty}} \cap B^d_{t_{-\infty}} &= B^d_{s_{\pm\infty}} \cap B^d_{t_{\pm\infty}} +(B^d_{s_{+\infty}} \cap B^d_{t_{-\infty}}) \cap (B_{s_{\pm\infty}} + B_{t_{\pm\infty}}) \\ &= B^d_{s_{\pm\infty}} \cap B^d_{t_{\pm\infty}} +B_{s_{-\infty}} \cap B^d_{t_{-\infty}}  + B^d_{s_{+\infty}} \cap B_{t_{+\infty}},
		\end{align*}
		so after intersecting with $B_{s_\cF < t_\cF}$ and then adding $B_{t_{+\infty}} \cap B_{s_{+\infty}}^d + B_{s_{-\infty}} \cap B_{t_{-\infty}}^d$ to both sides of the equality,  we get 
		\begin{align*}
			B_{t_{+\infty}} \cap B_{s_{+\infty}}^d &+ B_{s_{-\infty}} \cap B_{t_{-\infty}}^d + B_{s_{+\infty}}^d \cap B_{t_{-\infty}}^d \cap B_{s_\cF < t_\cF} \\& = B_{t_{+\infty}} \cap B_{s_{+\infty}}^d + B_{s_{-\infty}} \cap B_{t_{-\infty}}^d + B_{t_{\pm \infty}}^d \cap B_{s_{\pm \infty}}^d \cap B_{s_{\cF} < t_{\cF}} \\& = B_{s < t}. \qedhere
		\end{align*}
	\end{proof}
	
	\section{Type I improper integrals}
	In this section we will define improper integrals of the first kind, which are generalisations of classical expressions of the form $\int_{-\infty}^{\infty} f(x)dx$, $\int_{a}^{\infty} f(x)dx$, and $\int_{-\infty}^{b} f(x)dx$.
	
	Let $s, t \in E^{si}$ with $s \leq t$. We define the closed interval
	\[ [s, t] := \{x \in E^{si} \colon s \leq x \leq t \}, \]
	and set $[s, t]_E := [s, t] \cap E$. Further, we will define what it means to integrate a function over $[s, t]_E$. Note that $[s, t]_E \neq \varnothing$ implies $s \in E^i$ and $t \in E^s$.
	
	\begin{definition}
		Suppose $s \in E^i$ and $t \in E^s$ with $s \leq t$ and $[s, t]_E \neq \varnothing$. Suppose $f \colon [s, t]_E \to E$ is order bounded and locally band preserving. We say $f$ is \emph{integrable} from $s$ to $t$ if $f$ is integrable on $[a, b]$ for all $a, b \in [s, t]_E$ with $a \leq b$ and $\lim_{a \downarrow s, b \uparrow t} \int_a^b f(x) dx$ exists. We write
		\[ \int_s^t f(x) dx := \lim_{a \downarrow s, b \uparrow t} \int_a^b f(x) dx. \]
	\end{definition}
	
	We can extend this definition as in the proper integral case to incomparable elements as follows.
	
	\begin{definition}
		Let $s, t \in E^{si}$ be such that $s \wedge t \in E^i$, $s \vee t \in E^s$, and $[s \wedge t, s \vee t]_E \neq \varnothing$. Suppose $f \colon [s \wedge t, s \vee t]_E \to E$ is order bounded and locally band preserving. We say $f$ is \emph{integrable} from $s$ to $t$ if $f$ is integrable from $s \wedge t$ to $s \vee t$. In this case, we define
		\[ \int_s^t f(x) dx := \bP_{s < t} \left( \int_{s \wedge t}^{s \vee t} f(x) dx \right) - \bP_{t < s} \left( \int_{s \wedge t}^{s \vee t} f(x) dx \right). \]
	\end{definition}
	
	An application of the improper integrals of type I is an analogue of the integral test for series convergence.
	
	\begin{proposition}
		Suppose $E$ is a Dedekind complete $f$-algebra with multiplicative unit $e$. Let $f \colon [e, + \infty_E]_E \to E$ be locally band preserving, positive, and monotone decreasing. Then $\sum_{n = 1}^\infty f(n e)$ converges if and only if $\int_e^{+ \infty_E} f(x) dx$ exists.
	\end{proposition}
	\begin{proof}
		Since $f$ is locally band preserving, monotone, and order bounded, $f$ is integrable on $[e, ne]$ for every $n \in \bN$ by \cite[Proposition 4.7]{paper2}. For any $n \in \bN$, we have that
		\[ \int_{ne}^{(n+1)e} f(x) dx \leq \int_{ne}^{(n+1)e} f(n e) dx = f(ne). \]
		For $n \geq 2$, we have
		\[ f(ne) = \int_{(n - 1)e}^{n e} f(n e) dx \leq \int_{(n - 1) e}^{n e} f(x) dx. \]
		Therefore, for $N \in \bN$,
		\[ \int_{e}^{(N+1) e} f(x) dx \leq \sum_{n = 1}^N f(n e) \leq f(e) + \int_{e}^{N e} f(x) dx. \]
		It follows that $\lim_{n \to \infty} \int_{e}^{ne} f(x) dx$ exists if and only if $\sum_{n = 1}^\infty f(n e)$ exists. What remains to be shown is that $\lim_{n \to \infty} \int_{e}^{ne} f(x) dx$ exists if and only if $\int_e^{+ \infty_E} f(x) dx$ exists.  For any $b \in E$ with $b \geq e$ and $n \in \bN$,
		\[ 0 \leq \int_e^b f(x) dx - \int_e^{b \wedge ne} f(x) dx = \int_{b \wedge ne}^b f(x) dx \leq f(e) (b - (b \wedge ne)). \]
		Therefore,
		\[ \int_e^{b \wedge ne} f(x) dx \to \int_e^b f(x) dx \]
		as $n \to \infty$ since $e$ is a weak order unit. If $$L := \lim_{n \to \infty} \int_{e}^{ne} f(x) dx = \sup_{n \in \bN} \int_e^{ne} f(x) dx$$ exists, then
		\[ \int_e^{b \wedge ne} f(x) dx \leq \int_e^{ne} f(x) dx \leq L, \]
		and thus by taking the limit as $n \to \infty$, $\int_e^b f(x) dx \leq L$. Therefore, \[\int_e^{+ \infty_E} f(x) dx = \sup_{b \geq e} \int_e^b f(x) dx = L.\qedhere\]
	\end{proof}
	
	\begin{example}
		Let $E$ be a Dedekind complete unital $f$-algebra and let $r$ be an invertible element. Define $f \colon [r, + \infty_E]_E \to E$ by $f(x) = x^{-2}$ and $F \colon [r, + \infty_E]_E \to E$ by $F(x) = -x^{-1}$. Both $f$ and $F$ are well-defined---since if $x \geq r$ in $E$, then $x$ is invertible---and locally band preserving.
		
		A net $(x_\alpha)$ of invertible elements converging in order to an invertible element $x$ satisfies $x_\alpha^{-1} \to x^{-1}$ if and only if $\{\abs{x_\alpha}^{-1} \colon \alpha \geq \alpha_0\}$ is order bounded for some $\alpha_0$, see \cite[Remark 3.11]{differentiation}. For every $x \in [r, + \infty_E] \cap E$, we have that $0 \leq x^{-1} \leq r^{-1}$. It follows that $f$ and $F$ are order continuous and order bounded.
		
		By \cite[Lemma 3.14]{differentiation}, $F$ is order differentiable with $F' = f$. For any $r \leq b \in E$, $f$ is integrable on $[r, b]$ by \cite[Proposition 4.7]{paper2} since it is order bounded, monotone and locally band preserving.  By the Fundamental Theorem of Calculus \cite[Theorem 6.3]{paper2},
		\[ \int_r^b f(x) dx = F(b) - F(r) = r^{-1} - b^{-1}. \]
		
		We will show that as $r \leq b \uparrow + \infty_E$, we have $b^{-1} \downarrow 0$. Let $l := \inf_{b \geq r} b^{-1}$. For every $n \in \bN$, $r \leq n r$ and hence $l \leq (nr)^{-1} = n^{-1} r^{-1}$. Therefore, $l = 0$. It follows that
		\[ \int_r^{+ \infty_E} f(x) dx = r^{-1}. \]
	\end{example}
	
	\section{Type II improper integrals}
	Type II improper integrals allow us to integrate functions that are unbounded or undefined at finitely many points in the order interval $[a,b]$. We explain how to construct these integrals in this setting.
	
	\begin{definition}
		Let $f \colon (a, b] \to E$ be locally band preserving. Suppose for every $c \in (a, b]$, $f$ is order bounded and integrable on $[c, b]$. If the limit exists, we define
		\[ \int_a^b f(x) dx := \lim_{c \downarrow a, c \in (a, b]} \int_c^b f(x) dx. \]
	\end{definition}
	
	\begin{definition}
		Let $f \colon [a, b) \to E$ be locally band preserving. Suppose for every $c \in [a, b)$, $f$ is order bounded and integrable on $[a, c]$. If the limit exists, we define
		\[ \int_a^b f(x) dx := \lim_{c \uparrow b, b \in [a, b)} \int_a^c f(x) dx. \]
	\end{definition}
	
	\begin{definition}
		Let $f \colon [a, c) \cup (c, b] \to E$ be locally band preserving. Suppose for every $d_1 \in [a, c)$ and $d_2 \in (c, b]$ that $f$ is order bounded and integrable on both $[a, d_1]$ and on $[d_2, b]$. If the limit exists, we define
		\[ \int_a^b f(x) dx := \lim_{d_1 \uparrow c, d_2 \downarrow c} \left( \int_a^{d_1} f(x) dx + \int_{d_2}^b f(x) dx \right). \]
	\end{definition}
	
	
	\begin{example}
		Let $E$ be a universally complete $f$-algebra with multiplicative unit $e$. For $x \in E$, we define the generalised inverse $x^*$ of $x$ to be the multiplicative inverse of $x$ in the again universally complete band $B_x$. Note that $x$ is invertible in $B_x$ as it is a weak order unit. Define $f, F \colon [0, e] \to E$ by
		\[ f(x) := (\sqrt{x})^* \quad \text{and} \quad F(x) := 2 \sqrt{x}. \]
		Let $0 \ll r \ll e$. Then $F$ is locally band preserving, order continuous on $[r, e]$ and order differentiable on $(r, e)$ with $F' = f$. Also, $f$ is locally band preserving and uniformly continuous on $[r, e]$ and hence order bounded and integrable by \cite[Theorem 5.2]{paper1} and \cite[Proposition 4.8]{paper2}. Therefore,
		\[ \int_r^e f(x) dx = F(e) - F(r) = 2e - 2\sqrt{r}. \]
		Since as $r \downarrow 0$, we have that $\sqrt{r} \downarrow 0$, it follows that
		\[ \int_0^e f(x) dx = 2e. \]
	\end{example}
	
	\section{Integrating power series}
	
	Suppose in this section that $E$ is universally complete and let $e$ denote the multiplicative unit. We discuss integration of power series and show that this is done term-by-term, but first some preparatory results are needed.
	
	The following theorem is a generalisation of the analogue of the root test given in \cite[Theorem 3.6(i)]{series}, where we do away with the requirement that $(\abs{a_n}^{\frac{1}{n}})_{n \in \bN}$ be order bounded since suprema can now be taken in $E^s$. The proof is essentially still the same and hence omitted.
	
	\begin{theorem} \label{l:improved_root_test}
		Let $\sum_{n = 0}^\infty a_n$ be a series in $E$. If $\limsup_{n \to \infty} \abs{a_n}^{\frac{1}{n}}$ exists in $E$ and $\limsup_{n \to \infty} \abs{a_n}^{\frac{1}{n}} \ll e$ then the series $\sum_{n = 0}^\infty a_n$ converges absolutely in order.
	\end{theorem}

	\begin{lemma}\label{L:r in Omega}
		Let $f(x) := \sum_{n = 0}^\infty a_n (x-c)^n$ with spectrum of convergence $\Omega$ and radius of convergence $\rho$. If $r \in E^+$ such that $r \ll \rho$, then $r \in \Omega$.
	\end{lemma}
	\begin{proof}
		Let $L := \limsup_{n \to \infty} \abs{a_n}^{\tfrac{1}{n}}$ in $E^s$, then $\rho$ is the generalized inverse of $L$ in $E^s$ and $\rho L = \rho_\cF L_\cF = \bP_{\rho_\cF}(e)$ by \cite[Theorem 5.20]{differentiation}. Since $0 \leq r \leq \rho$, it follows from \Cref{p:properties_multiplication_sup_comp}(iv) that
		\[ 0 \leq rL \leq \rho L = \bP_{\rho_\cF} (e). \]
		Thus, $rL \in E$ with $0 \leq rL \leq e$. By \Cref{p:properties_of_ll}(ii) and \Cref{p:wou_gg_zero}, $\rho$ is a weak order unit of $E$ since $\rho \gg r \geq 0$, so that $E = B_{\rho_{+\infty}} \oplus B_{\rho_\cF}$. Now $\bP_{\rho_\cF}(r) \ll_{B_{\rho_\cF}} \bP_{\rho_\cF}(\rho) = \bP_{\rho_\cF}(\rho_\cF)$ by \Cref{p:properties_of_ll}(v), so that $\bP_{\rho_\cF}(\rho - r) = \bP_{\rho_\cF} (\rho_\cF - r)$ is a weak order unit of $B_{\rho_\cF}$. Since $L_\cF$ is an invertible element of $B_{\rho_\cF}$, we have that $\bP_{\rho_\cF}(L_\cF(\rho_\cF - r))$ is also a weak order unit of $B_{\rho_\cF}$. Therefore,
		\begin{align*}
			e - rL 
			& = \bP_{\rho_\infty}(e - rL) + \bP_{\rho_\cF}(e - rL) \\
			& = \bP_{\rho_\infty}(e) + \bP_{\rho_\cF}(L_\cF (\rho_\cF - r))
		\end{align*}
		is a weak order unit of $E$. Thus, $rL \ll e$. Let $x \in [c - r, c + r]$. Then $\abs{x - c} \leq r$, and hence
		\[ \limsup_{n \to \infty} \left( \abs{a_n}\abs{x - c}^n \right)^{\tfrac{1}{n}} = \limsup_{n \to \infty} (\abs{a_n}^{\tfrac{1}{n}} \abs{x - c} ) = L \abs{x - c} \ll e. \]
		
		Therefore, $f(x)$ converges absolutely in order on $[c - r, c + r]$ by \Cref{l:improved_root_test}. For each $m \in \bN$ and all $x \in [c - r, c + r]$,
		\[ \left|{\sum_{n = 0}^m a_n(x-c)^n - f(x)}\right| \leq \sum_{n = m+1}^\infty \abs{a_n} r^n \downarrow_m 0. \]
		Thus, $f(x)$ converges uniformly in order on $[c - r, c + r]$ and hence $r \in \Omega$.
	\end{proof}
	
	\begin{lemma}
		Let $f(x) := \sum_{n = 0}^\infty a_n (x-c)^n$ with spectrum of convergence $\Omega$ and $r \in \Omega$. Then $f$ is uniformly continuous on $[c - r, c + r]$ and hence integrable.
	\end{lemma}
	\begin{proof}
		Let $f_N(x) := \sum_{n = 0}^\infty a_n (x - c)^n$. Then each $f_N$ is uniformly continuous, and $f_N \to f$ uniformly on $[c - r, c + r]$. Therefore, $f$ is uniformly continuous on $[c - r, c + r]$
	\end{proof}
	
	By \cite[Theorem 6.2]{differentiation}, the power series $f(x) := \sum_{n = 0}^\infty a_n (x - c)^n$ and $g(x) := \sum_{n = 0}^\infty \frac{1}{n+1}a_n (x - c)^{n+1}$ have the same radius of convergence $\rho$. If $\rho$ dominates a positive invertible element of $E$, then $g'(x) = f(x)$ for all $x \in (c - \rho, c + \rho)$ by \cite[Theorem 6.7]{differentiation}. Consequently, we call the power series $g(x)$ the \emph{formal antiderivative} of the power series $f(x)$.
	
	\begin{proposition}
		Let 
		\[
		f(x) := \sum_{n = 0}^\infty a_n (x-c)^n \quad \mbox{and}\quad g(x) := \sum_{n = 0}^\infty \tfrac{1}{n+1}a_n (x - c)^{n+1} 
		\]
		with the same radius of convergence $\rho$ that dominates a positive invertible element of $E$, say $s$. For $a, b \in (c - \rho, c + \rho) \cap E$, we have that $f$ is integrable on $[a \wedge b, a \vee b]$ and
		\[ \int_a^b f(x) dx = g(b) - g(a). \]
	\end{proposition}
	\begin{proof}
		Let $a, b \in (c - \rho, c + \rho) \cap E$. Let $r := \abs{a - c} \vee \abs{b - c} \vee \frac12 s$. Then $r$ is a weak order unit of $E$ and $r \ll \rho$ by \Cref{p:properties_of_ll}(iv), so that $r \in \Omega$ by Lemma~\ref{L:r in Omega}. Now $f$ is locally band preserving, uniformly continuous and hence order bounded and integrable on $[c - r, c + r]$ by \cite[Theorem 5.2]{paper1} and \cite[Proposition 4.8]{paper2}. Also, $g$ is locally band preserving, order continuous on $[c - r, c + r]$ and order differentiable on $(c - r, c + r)$ with $g' = f$. Therefore, by \cite[Theorem 6.3]{paper1},
		\[ \int_a^b g(x) dx = f(b) - f(a). \qedhere \]
	\end{proof}
	
	\noindent\textbf{Acknowledgements.}
	The authors would like to express their gratitude to the European Union Erasmus+ ICM programme for facilitating productive visits between Leiden University and the University of Pretoria.
	
	\bibliographystyle{alpha}
	\bibliography{refs.bib}

@article {paper1,
    AUTHOR = {E. Kikianty and L. Naude and M. Roelands and C. Schwanke},
     TITLE = {Classical theorems from analysis for locally band preserving
              functions on {D}edekind complete {$\Phi $}-algebras},
   JOURNAL = {Positivity},
  FJOURNAL = {Positivity. An International Mathematics Journal Devoted to
              Theory and Applications of Positivity},
    VOLUME = {30},
      YEAR = {2026},
    NUMBER = {2},
     PAGES = {Paper No. 31, 26},
      ISSN = {1385-1292,1572-9281},
   MRCLASS = {46A40 (26A24)},
  MRNUMBER = {5059522},
       DOI = {10.1007/s11117-026-01188-6},
       URL = {https://doi.org/10.1007/s11117-026-01188-6},
}

@misc{paper2,
      title={The {R}iemann integral on {D}edekind complete $f$-algebras}, 
      author={E. Kikianty and L. Naude and M. Roelands and C. Schwanke},
      year={2026},
      eprint={2604.26623},
      archivePrefix={arXiv},
      primaryClass={math.FA},
      url={https://arxiv.org/abs/2604.26623}, 
}

@article {TroitskyPolavarapu,
    AUTHOR = {A.R. Polavarapu and V.G. Troitsky},
     TITLE = {A representation of sup-completion},
   JOURNAL = {Proc. Amer. Math. Soc.},
  FJOURNAL = {Proceedings of the American Mathematical Society},
    VOLUME = {152},
      YEAR = {2024},
    NUMBER = {8},
     PAGES = {3403--3411},
      ISSN = {0002-9939,1088-6826},
   MRCLASS = {46A40},
  MRNUMBER = {4767271},
MRREVIEWER = {Christos\ E.\ Kountzakis},
       DOI = {10.1090/proc/16796},
       URL = {https://doi-org.uplib.idm.oclc.org/10.1090/proc/16796},
}

@book {AliprantisBurkinshaw,
    AUTHOR = {C.D. Aliprantis and O. Burkinshaw},
     TITLE = {Locally solid {R}iesz spaces with applications to economics},
    SERIES = {Mathematical Surveys and Monographs},
    VOLUME = {105},
   EDITION = {Second},
 PUBLISHER = {American Mathematical Society, Providence, RI},
      YEAR = {2003},
     PAGES = {xii+344},
      ISBN = {0-8218-3408-8},
   MRCLASS = {46A40 (46N10 91B50)},
  MRNUMBER = {2011364},
MRREVIEWER = {Pedro\ Jim\'enez Guerra},
       DOI = {10.1090/surv/105},
       URL = {https://doi.org/10.1090/surv/105},
}

@book {Donner,
    AUTHOR = {K. Donner},
     TITLE = {Extension of positive operators and {K}orovkin theorems},
    SERIES = {Lecture Notes in Mathematics},
    VOLUME = {904},
 PUBLISHER = {Springer-Verlag, Berlin-New York},
      YEAR = {1982},
     PAGES = {xii+182},
      ISBN = {3-540-11183-2},
   MRCLASS = {46A22 (41A36 46B30 46E05 47B99)},
  MRNUMBER = {653635},
MRREVIEWER = {I.\ Namioka},
}

@book{babyzaanen,
author = {A.C. Zaanen},
year = {1997},
title = {Introduction to Operator Theory in Riesz Spaces},
publisher = {Springer}
}

@book{zaanen1,
author = {W.A.J. Luxemburg and A.C. Zaanen},
year = {1971},
title = {Riesz Spaces I},
publisher = {North-Holland Publishing Co.}
}

@book{zaanen2,
author = {A.C. Zaanen},
year = {1983},
title = {Riesz Spaces II},
publisher = {North-Holland Publishing Co.}
}

@article{series,
title = {Series and power series on universally complete complex vector lattices},
author = {M. Roelands and C. Schwanke},
journal = {J. Math. Anal. Appl.},
volume = {473},
number = {2},
pages = {680--694},
year = {2019},
publisher = {Elsevier}
}

@phdthesis{depagter,
    author = {B. de {P}agter},
    sortname = {Pagter},
    title = {$f$-algebras and Orthomorphisms},
    school  = {Rijksuniversiteit te Leiden},
    year    = {1981},
}

@article {differentiation,
    AUTHOR = {Roelands, M. and Schwanke, C.},
     TITLE = {Differentiable, holomorphic, and analytic functions on complex
              {$\Phi $}-algebras},
   JOURNAL = {J. Math. Anal. Appl.},
  FJOURNAL = {Journal of Mathematical Analysis and Applications},
    VOLUME = {541},
      YEAR = {2025},
    NUMBER = {1},
     PAGES = {Paper No. 128671, 26},
      ISSN = {0022-247X,1096-0813},
   MRCLASS = {30G99 (32A65 46A40)},
  MRNUMBER = {4776394},
       DOI = {10.1016/j.jmaa.2024.128671},
       URL = {https://doi.org/10.1016/j.jmaa.2024.128671},
}

@article {Azouzi,
    AUTHOR = {Y. Azouzi},
     TITLE = {Completeness for vector lattices},
   JOURNAL = {J. Math. Anal. Appl.},
  FJOURNAL = {Journal of Mathematical Analysis and Applications},
    VOLUME = {472},
      YEAR = {2019},
    NUMBER = {1},
     PAGES = {216--230},
      ISSN = {0022-247X,1096-0813},
   MRCLASS = {46A40},
  MRNUMBER = {3906369},
MRREVIEWER = {Christopher\ Schwanke},
       DOI = {10.1016/j.jmaa.2018.11.019},
       URL = {https://doi.org/10.1016/j.jmaa.2018.11.019},
}

@article {AzouziSI1,
title = {The sup-completion of a Dedekind complete vector lattice},
journal = {Journal of Mathematical Analysis and Applications},
volume = {506},
number = {2},
pages = {125651},
year = {2022},
issn = {0022-247X},
doi = {https://doi.org/10.1016/j.jmaa.2021.125651},
url = {https://www.sciencedirect.com/science/article/pii/S0022247X21007307},
author = {Y. Azouzi and Y. Nasri},
}
	
\end{document}